\documentclass[reqno]{amsart}
\usepackage{amsmath,amssymb}
\usepackage{graphics,epsfig,color}
\usepackage[mathscr]{euscript}
\theoremstyle{plain}
\newtheorem{theorem}{Theorem}[section]
\newtheorem{proposition}[theorem]{Proposition}
\newtheorem{lemma}[theorem]{Lemma}

\theoremstyle{definition}

\theoremstyle{remark}

\numberwithin{equation}{section}
\numberwithin{theorem}{section}

\renewcommand{\epsilon}{\varepsilon}
\renewcommand{\tilde}{\widetilde}

\renewcommand{\div}{\mathop{\rm div}\nolimits}

\definecolor{light}{gray}{.9}

\title[Stochastic monotonicity from an Eulerian viewpoint]{Stochastic monotonicity from an Eulerian viewpoint }

\author[D.\ Gabrielli]{Davide Gabrielli}
\address{Davide Gabrielli \hfill\break \indent
  DISIM, University of L'Aquila
  \hfill\break\indent
  Via Vetoio,   67100 Coppito, L'Aquila, Italy}
\email{gabriell@univaq.it}

\author[I.G.\ Minelli]{Ida Germana Minelli}
\address{Ida Germana Minelli \hfill\break \indent
  DISIM, University of L'Aquila
  \hfill\break\indent
  Via Vetoio,   67100 Coppito, L'Aquila, Italy}
\email{ida.minelli@univaq.it}

\begin{document}

\begin{abstract}
Stochastic monotonicity is a well known partial order
relation between probability measures defined on the same partially ordered set.
Strassen Theorem establishes equivalence between stochastic monotonicity
and the existence of a coupling compatible with respect to the partial order. We consider the case of a countable set and introduce the class of \emph{finitely decomposable flows} on a directed acyclic graph associated to the partial order. We show that a probability measure stochastically dominates another probability measure if and only if
there exists a finitely decomposable flow  having divergence given by the difference of the two measures. We illustrate the result with some examples. In fluid theory the Lagrangian description follows the trajectories of the particles while the Eulerian one
observes the local flow. A coupling gives a Lagrangian description of the transference plan of mass while a flow gives an Eulerian one.

\bigskip


\end{abstract}

\maketitle
\thispagestyle{empty}

\section{Introduction}
Given a partially ordered set (from now on, \emph{poset}) $V$ there is a naturally induced partial order relation on the set of probability measures on $V$, usually called \emph{stochastic monotonicity}.
Given two probability measures $\mu_1, \mu_2$, we say that $\mu_2$ stochastically dominates $\mu_1$, and write $\mu_1\preceq \mu_2$, if the expectation of any bounded increasing function with respect to the measure $\mu_1$ is less or equal than the expectation with respect to $\mu_2$.

Strassen Theorem \cite{L,LL,St} is an important and powerful result in Probability Theory. It states that $\mu_2$ stochastically dominates $\mu_1$ if and only if there
exists a coupling between the two measures that gives zero weight to pairs of elements
not increasingly ordered.

We consider the case of a countable poset $V$ and show a new equivalent statement for stochastic domination, based on a graph structure associated to $V$.
Indeed, the partial order structure of a countable poset can be described in terms
of an acyclic directed graph $(V,E)$. On such a graph, it is possible to define a \emph{flow}, which  is an assignment of a positive weight, representing the amount of mass flown,
to each directed edge. The \emph{divergence} of the flow on a vertex is defined as the difference between the amount of mass flown outside and the amount of mass flown into the vertex. We will define the class of \emph{finitely decomposable} flows on $(V,E)$, i.e., flows that can be decomposed as a summable superposition of ''elementary flows'' associated to finite self avoiding paths on the graph (see section \ref{NeR} for the definition), and
we will prove that $\mu_2$ stochastically dominates $\mu_1$ if and only if there exists a finitely decomposable flow  having as divergence the difference between the two measures. This statement may be reformulated intuitively by saying that one measure stochastically dominates another one when it is possible to transform this second measure into the first one by moving mass according to the partial order structure. We emphasize that with this new formulation
stochastic domination is shown to be equivalent to the existence of a flow on a directed graph encoding all the information about the partial order.
This allows to connect directly monotonicity results to the geometry of the underlying partial order. In particular, in section \ref{appl} we will discuss a dual problem for which the homological structure of the directed graph turns out to be a relevant characteristic.

We have therefore the 3 equivalent statements: (1) $\mu_1 \preceq \mu_2$; (2) there exists a finite decomposable flow
having divergence $\mu_1-\mu_2$: (3) there exists a compatible coupling. The content of Strassen Theorem is that
(1) $\Leftrightarrow$ (3). We show how to prove all the remaining implications. We show that (1) $\Leftrightarrow$ (2) can be obtained using Farkas Lemma \cite{S} in the finite case and a suitable infinite dimensional version of such Lemma in the infinite case. The equivalence (2)$\Leftrightarrow$ (3) can be obtained
using ideas from the theory of mass transportation \cite{RR,Sant} where an equivalence with a continuous problem of flows also appears \cite{B,Sant}. In particular we give a constructive proof obtained by an algorithmic construction that associates a coupling to any finite acyclic flow \cite{PS} and is a discrete version of a construction due to S.K. Smirnov on bounded domains of $\mathbb R^n$ \cite{SS}. The countable infinite case is obtained extending the algorithmic construction , with a limiting argument, to the class of finitely decomposable flows.

There is a kind of hierarchical structure between the statements (1), (2) and (3) and the proof of any implication (i) $\Rightarrow$ (j) require a difficult argument when $i<j$ while a simple construction is enough when $i>j$. The proof of the equivalence
between statements (1), (2) and (3) can be obtained by one of the two cycles of implications : (1) $\Rightarrow$ (3) $\Rightarrow$ (2) $\Rightarrow$ (1) or (1) $\Rightarrow$ (2) $\Rightarrow$ (3) $\Rightarrow$ (1).

We present the proofs of all implications (apart the ones of the classic Strassen Theorem) since they are interesting in themselves and reflect the geometric structure behind.

Our result is similar in spirit to the equivalence between the  Monge-Kantorovich problem with cost $|x-y|$ and a
minimal flow problem proposed by Beckmann \cite{B}. This equivalence in illustrated in Chapter 4 of \cite{Sant}.
The reason of this equivalence is in the fact that the cost in the Monge-Kantorovich problem does not depend on the details of the transference plan but only on the flow of mass locally observed. The same happens for stochastic order. In particular our result is the counterpart for stochastic monotonicity of Theorem 4.6 in \cite{Sant} for mass transportation.

We first prove our result in the finite case, when the partial order structure can be encoded by a minimal directed acyclic graph called \emph{Hasse diagram}.
The result in the finite case is implicit in \cite{HW} where it is proved using the theory of convex games. There are also similar ideas and statements in \cite{M} (and references therein) obtained by duality. A clear formulation in terms of a flow on network problem is however missing and moreover we extend the result to the countable infinite case. Item $(v)$ of Theorem 1 of \cite{KKO} can be seen as a very special case of our formulation.

At the beginning of section \ref{NeR} we give a more detailed overview of the basic
ideas and constructions behind our results. In section \ref{appl} we discuss some examples.
Even if the examples are simple they are important to point out the change of perspective
with respect to the usual approach. The proofs of monotonicity are indeed obtained with computations that are different from the usual ones.

\smallskip
The structure of the paper is the following.\\
In section \ref{NeR} we fix notation and state the main result of the paper.
In section \ref{Strassentot} we prove Theorems \ref{ilteorema} and \ref{ilteorema-infinito}.
In section \ref{PPP} we prove some auxiliary Propositions \ref{proposizione} and \ref{proposizione3}.
In section \ref{appl} we discuss some examples.

\section{Preliminaries, notation and main results}
\label{NeR}

In this section we discuss the general framework, introduce notation and
state our main results.
We start with a short introductive illustration of our results. We discuss informally the intuitive idea and the novelty with respect to the classic statement.

\subsection{Preliminaries}
We start recalling again the Strassen Theorem (see for example Section 1 of Chapter IV of \cite{LL}). It says that the statements
$(1)\ \mu_1\preceq \mu_2 $, and $(3)$ there exists a compatible coupling $\rho$ between $\mu_1$ and $\mu_2$, are indeed equivalent.

A coupling $\rho$ determines a transportation of mass according to which the initial distribution of mass $\mu_1$ is transformed into the final distribution of mass $\mu_2$. The value $\rho(x,y)$ fix the amount of mass that has to be moved from site $x$ to site $y$. Another possible description of the transportation is obtained looking at the channels through which the mass can flow and recording the amount of mass flown, without taking care of origin and destination. This is obtained giving a flow $Q$ for which $Q(x,x')$ is the amount of mass flown across the channel $(x,x')$.
This second description is less detailed since you have not a complete view of the transference plan and several couplings can correspond to the same flow. These two perspectives are similar to the Lagrangian and the Eulerian point of view in fluid theory.

In the case of a countable partial order the channels are naturally individuated by the edges of an acyclic directed graph $(V,E)$ determining the partial order. For a flow on a graph there is a natural definition of discrete divergence. If initially we have a distribution of mass $\mu_1$ and we let flow mass according to a flow such that $\div Q=\mu_1-\mu_2$ at the end we obtain the distribution of mass
$\mu_2$. Note that the existence of a flow satisfying $\div Q=\mu_1-\mu_2$ is not obvious since the edges of the graph are oriented and the mass can flow only in one direction. We can formulate now a natural statement: $(2)$ there exists a flow on a directed acyclic graph determining the partial order such that $\div Q=\mu_1-\mu_2$. The main result of this paper is that the statements $(1)$, $(2)$ and $(3)$ are all equivalent.

To prove that $\mu_1\preceq \mu_2$ we can then use the statement $(3)$ constructing a compatible coupling. This is a collection of
$|V|\times |V|$ numbers with $2|V|+1$ constraints. We can however also use statement $(2)$ exhibiting a flow on $(V,E)$ such that $\div Q=\mu_1-\mu_2$. This is a simpler object determined by a collection of $|E|$ numbers with $|V|$ constraints.

\subsection{Digraphs and posets}

We consider a countable set $V$. A directed graph, called shortly a \emph{digraph}, with vertices set $V$ is a pair $(V,E)$
where $E\subset V\times V$ is a collection of directed edges.
We assume that there are not edges of the type $(x,x)$.
A directed path $\gamma$ from $x\in V$ to $y\in V$ is
a sequence of vertices $\gamma:=(x_0,\dots,x_n)$ such that $x_0=x$, $x_n=y$ and $(x_i,x_{i+1})\in E$, for $i=0,\dots,n-1$.
The integer $n$ is the length of the directed path and is denoted also by $|\gamma|$.
If there exists an $i$ such that $(u,v)=(x_i,x_{i+1})$ we write $(u,v)\in \gamma$. Given a subset $S\subseteq V$, if $x_i\in S$ for any $i$ we write $\gamma \subseteq S$. We call $\gamma^-:=x_0$ the starting point
of the path and $\gamma^+:=x_n$ its final point. A directed cycle is a directed path for which
$x_0=x_n$. Given two paths $\gamma=(x_0,\dots , x_n)$ and $\gamma'=(x'_0,\dots , x'_k)$
such that $x_n=x'_0$ we denote by
\begin{equation}
\gamma \star \gamma':=(x_0,\dots , x_n , x'_1,\dots , x'_k)\,,
\label{conc}
\end{equation}
the path given by their concatenation. A path is called self-avoiding if $x_i\neq x_j$ when $i\neq j$.
A digraph containing no directed cycles is called a directed acyclic graph.
Given a digraph $(V,E)$ we call $(V,\mathcal E)$ the un-directed graph with
edges
$$
\mathcal E:=\left\{\{x,y\}\,:\, (x,y)\in E\ \textrm{or}\ (y,x)\in E\right\}\,.
$$
\smallskip
Given a digraph $(V,E)$ we can construct a new digraph $(V,\overline E)$ called its \emph{transitive closure}. A pair
$(x,y)\in \overline E$ if and only if there exists a directed path form $x$ to $y$.
When $|V|<+\infty$ and $(V,E)$ is
an acyclic digraph  we can define also a new directed acyclic graph $(V,\underline E)$ that is called its
\emph{transitive reduction}: it is the minimal acyclic digraph having the same transitive
closure as $(V,E)$, i.e., such that for any digraph $(V,F)$ with
$(V,\overline E)=(V,\overline F)$ we have $\underline E\subseteq F$. When the original digraph $(V,E)$ is acyclic
and $|V|<+\infty$, it can be shown that $(V,\underline E)$ is uniquely determined (see \cite{D} section 4.3).\\

A partial order relation $\leq$ on $V$ is a subset $S\subseteq V\times V$
satisfying the properties of \emph{reflexivity, antisymmetry and transitivity}. When $(x,y)\in S$ we write $x\leq y$ and
the pair $(V,\leq)$ is called a partially ordered set or simply a \emph{poset}.
Then, as can be easily checked, if we set $\overline E=S\setminus\{(x,x); x\in V\}$, the pair $(V,\overline E)$
gives an acyclic digraph whose transitive closure coincides with itself. On the other hand, any acyclic digraph $(V,E)$ induces
a partial order on $V$ through the relation
$x\leq y\ \Leftrightarrow\ x=y\ \mbox{or}\ (x,y)\in \overline{E}$. So, a poset can be described with an acyclic digraph.\\

Note that such a description is not unique, since different digraphs can have the same transitive closure. However, when $V$ is finite, it is uniquely identified the transitive reduction $(V, \underline{E})$ and it is called the \emph{Hasse diagram} of the poset.\\
When $|V|=+\infty$, it is not always possible to define the transitive reduction (think, as an example, to the set of rationals); nevertheless, any acyclic digraph has a well defined
transitive closure and consequently it determines a partial order on $V$. Any countable infinite poset can be described in terms of an acyclic digraph.\\

\subsection{Couplings and flows}
\label{cfdvf}
Let  $\mu_1$ and $\mu_2$ be two probability measures on a poset $(V, \leq )$.\\
A \emph{coupling} between $\mu_1$ and $\mu_2$ is a probability measure $\rho$ on $V\times V$ such that
$$
\left\{
\begin{array}{ll}
\sum_{y\in V}\rho(x,y)=\mu_1(x) \,, & \forall x\in V\,,\\
\sum_{x\in V}\rho(x,y)=\mu_2(y) \,, & \forall y\in V\,.
\end{array}
\right.
$$
We say that a coupling $\rho$ is \emph{compatible} with the partial order $\leq$
if
$$
\rho\{(x,y)\,:\, x\leq y\}=1\,.
$$
We say that $\mu_2$ \emph{stochastically dominates} $\mu_1$ with respect to the partial order $\leq$
and write $\mu_1\preceq \mu_2$ if, for any bounded increasing function $f: V\longrightarrow \mathbb{R}$ (i.e., a function such that $f(x)\leq f(y)$ whenever $x\leq y$ in $V$) we have
\begin{equation*}
\mu_1(f)\leq \mu_2(f)
\end{equation*}
 where $\mu(f)=\mathbb{E}_\mu (f)$ denotes expectation with respect to $\mu$.\\

Let $(V,E)$ be a digraph.\\
A \emph{flow} on $(V,E)$ is a map
$Q: E\to \mathbb R^+$. The \emph{divergence} of $Q$ at $x\in V$
is defined by
\begin{equation}
\div Q(x):=\sum_{y:(x,y)\in E}Q(x,y)-\sum_{y:(y,x)\in E}Q(y,x)\,.
\label{divQ}
\end{equation}
When $|V|=+\infty$ the divergence is not always well defined.
In this case we say that the divergence of a flow $Q$ exists and is given by \eqref{divQ} if both series
appearing in the r.h.s. of \eqref{divQ} are convergent for any $x\in V$. \\
We denote by $E(Q)$ the elements $(x,y)\in E$ such that $Q(x,y)>0$. We say that a flow is \emph{acyclic} if the digraph $(V, E(Q))$ is acyclic.
Given a directed path $\gamma=(x_0,\dots,x_n)$ on $(V,E)$, we associate to it the flow $Q_\gamma$ defined by
\begin{equation}\label{flowcycl}
Q_\gamma(x,y):=\left\{\begin{array}{ll}
1 & \textrm{if}\ (x,y)\in \gamma\,, \\
0 & \textrm{otherwise}\,.
\end{array}
\right.
\end{equation}
On the set of flows on a fixed digraph there is a natural partial order structure defined
by $Q\leq Q'$ if $Q(x,y)\leq Q'(x,y)$ for any $(x,y)\in E$.

\subsection{Finitely decomposable flows}
We say that a flow $Q$ on the digraph $(V,E)$ is \emph{finitely decomposable} if there exists a countable family
of finite self avoiding directed paths $\left\{\gamma_n\right\}_{n\in \mathbb N}$ and a  sequence
$\left\{q_n\right\}_{n\in \mathbb N}$ of weights  $q_n\geq 0$ with $\sum_n q_n<+\infty$,
such that
\begin{equation}\label{finite-decomp}
Q=\sum_nq_n Q_{\gamma_n}\,.
\end{equation}
Note that in an acyclic digraph any path is self avoiding.
If $|V|<+\infty$ then any flow is finitely decomposable since we have for example
$$
Q=\sum_{(x,y)\in E}Q(x,y)Q_{(x,y)}\,.
$$
A finitely decomposable flow is not necessarily summable since
we have
\begin{equation}
\sum_{(x,y)\in E}Q(x,y)=\sum_n q_n|\gamma_n|
\label{Fub}
\end{equation}
and the r.h.s. of \eqref{Fub} can be infinite.
Note that a finite decomposition \eqref{finite-decomp} of a finitely decomposable
flow induces naturally a finite positive measure on $\Gamma$ the countable set of all
finite self-avoiding paths on $(V,E)$. This is simply
\begin{equation}\label{fini-ind-mea}
\sum_n q_n \delta_{\gamma_n}\,,
\end{equation}
where $\delta$ is the delta Dirac measure.\\
Since the paths in \eqref{finite-decomp} are self avoiding, every single path $\gamma_n$ may
contribute just once to the outgoing or ingoing flux at a single site. This implies that
the divergence of a finitely
decomposable flow is well defined and it is given by
\begin{equation}
\div Q(x)=\sum_{\left\{n\,:\, \gamma_n^-=x\right\}}q_n-\sum_{\left\{n\,:\, \gamma_n^+=x\right\}}q_n\,.
\label{divQ-fd}
\end{equation}

\subsection{A third equivalent statement in Strassen Theorem}

We start with the finite case.
\begin{theorem}
Let $(V,\leq)$ be a finite poset and let $(V,\underline E)$ be the acyclic digraph
associated to its Hasse diagram. The following statements are equivalent.
\begin{enumerate}

\item $\mu_1\preceq \mu_2$\,,

\item there exists a flow $Q$ on $(V,\underline E)$ such that $\div Q= \mu_1-\mu_2$\,,

\item there exists a compatible coupling between $\mu_1$ and $\mu_2$\,.

\end{enumerate}
\label{ilteorema}
\end{theorem}
We recall that Strassen Theorem states the equivalence between stochastic domination $\mu_1\preceq \mu_2$ and the existence of a compatible coupling between $\mu_1$ and $\mu_2$.\  Statement $(2)$ shows that finding such a coupling is equivalent to solve a flow on network problem. In statement $(2)$ the digraph $(V,\underline E)$ could be
replaced by any acyclic digraph having the same transitive closure.

\smallskip
A result analogous to Theorem \ref{ilteorema} holds also when $|V|=+\infty$ but we need the additional assumption of finite decomposability of $Q$.
\\

\begin{theorem}
Let $(V,\leq)$ be a countable infinite partial order and let $(V,E)$ be a directed acyclic graph
such that its transitive closure $(V,\overline E)$ induces the partial order $\leq$.
The following statements are equivalent.
\begin{enumerate}

\item $\mu_1\preceq \mu_2$\,,

\item there exists a finitely decomposable flow $Q$ on $(V,E)$ such that\\
$\div Q= \mu_1-\mu_2$\,.

\item there exists a compatible coupling between $\mu_1$ and $\mu_2$\,.

\end{enumerate}
\label{ilteorema-infinito}
\end{theorem}

In general it is not easy to verify if a flow $Q$ is finitely decomposable, so we state a
sufficient and a necessary condition.\\
Let $(V,E)$ be an infinite digraph. An invading sequence of vertices $\{V_n\}_{n\in \mathbb N}$ is
a sequence of subsets $V_n\subseteq V$ such that $|V_n|<+\infty$, $V_n\subseteq V_{n+1}$ and
$\cup_n V_n=V$.
Given a flow $Q$ we say that it has
\emph{zero flux towards infinity} (see \cite{BFG} for the original definition and related results)
if there exists an invading sequence of vertices such that
$$
\lim_{n\to +\infty}\left(\sum_{x\in V_n, y\not\in V_n}Q(x,y)\right)=0\,.
$$
We have the following sufficient condition.
\begin{proposition}\label{proposizione}
Let $Q$ be a flow on an infinite acyclic digraph $(V,E)$ such that $\div Q=\mu_1-\mu_2$.
If $Q$ has zero flux towards infinity, then it is finitely decomposable.
\end{proposition}
Next proposition gives a necessary condition.
\begin{proposition}\label{proposizione3}
Let $Q$ be a flow on an infinite digraph $(V,E)$. If $Q$ is finitely decomposable then for all invading sequences
\begin{equation}
\lim_{n\to +\infty} \sup\Big\{Q(x,y)\,:\, \left\{x,y\right\}\cap \left\{V\setminus V_n\right\}\neq \emptyset\Big\}= 0\,.
\label{london}
\end{equation}
\end{proposition}

\section{Proofs of Theorems \ref{ilteorema} and \ref{ilteorema-infinito}}
\label{Strassentot}

We discuss the proofs of all the implications apart (1) $\Leftrightarrow$ (3)
that is the the content of the classic Strassen Theorem. The proofs of the Theorems can be obtained using one of the two cycles of implications : (1) $\Rightarrow$ (3) $\Rightarrow$ (2) $\Rightarrow$ (1) or (1) $\Rightarrow$ (2) $\Rightarrow$ (3) $\Rightarrow$ (1).
We give the proofs of all the implications since  they are  interesting  in themselves and give insight to the
geometric structures involved. We remark that the proof (2) $\Rightarrow$ (3) is constructive and using ideas from mass transportation theory shows how to construct a coupling starting from an acyclic finite decomposable flow.

\subsection{Proof of Theorem \ref{ilteorema}}\label{Strassen}

\subsubsection{(2) $\Rightarrow$ (1)}
Let $Q$ be a finite  flow such that $\div Q=\mu_1-\mu_2$. Then
by a discrete integration by parts we have
\begin{equation*}
\mu_2(f)-\mu_1(f)=-\sum_{x\in V}f(x)\div Q(x)=\sum_{(x,y)\in E}Q(x,y)\left(f(y)-f(x)\right)\geq 0\,.
\end{equation*}

\subsubsection{(1) $\Rightarrow$ (2)}

Farkas Lemma (see \cite{S} volume A section 5.4) states that, given an
$n\times m$ matrix $A$ and $b\in \mathbb R^n$, there exists $x\in (\mathbb R^+)^m$ such that $Ax=b$ if and only if
for any $y\in \mathbb R^n$ such that $A^Ty\in (\mathbb R^+)^m$ the inequality $y\cdot b \geq 0$ holds (where $\cdot$ denotes the Euclidean scalar product).
\smallskip
Let us consider the adjacency matrix $A$ of $(V,\underline E)$. It is a $|V|\times |\underline E|$ matrix whose rows and columns are labeled respectively with the vertices of $V$ and the edges of $\underline E$ and it is defined by fixing equal to $+1$ the element corresponding
to the row $x$ and the column $(x,y)$, and by fixing equal to $-1$ the element corresponding to the row $y$ and the column $(x,y)$.
All the remaining elements in the column $(x,y)$ are set equal to $0$.
With this definition, given a flow $Q$ we have
$$
\div Q(x)=AQ(x)\,.
$$
Moreover given a function $f:V\to \mathbb R$ we have that
$$
-A^Tf(x,y)=f(y)-f(x)\,.
$$
The function $f$ is increasing if and only if $A^Tf(x,y)\leq 0$ for any $(x,y)\in \underline E$. The result now follows
applying Farkas Lemma with the matrix $A$ coinciding with the adjacency matrix and taking the vector
$b=\mu_1-\mu_2$.\\

\subsubsection{(3)$\Rightarrow$ (2)}
Suppose that there exists a compatible coupling $\rho$ between $\mu_1$ and $\mu_2$.
If $x\leq y$ and $x\neq y$ there exists at least one directed path in $(V, \underline E)$
going from $x$ to $y$. Fix one of them arbitrarily and call it $\gamma_{(x,y)}$.
Recalling definition \eqref{flowcycl}, we construct the flow
\begin{equation}
Q:=\sum_{\left\{x,y\in V\,:\, x\neq y\right\}}\rho(x,y)Q_{\gamma_{(x,y)}}\,.
\label{Qdarho}
\end{equation}
This flow is finitely decomposable by definition and it
satisfies $\div Q=\mu_1-\mu_2$. Indeed, using \eqref{divQ-fd} and the fact that $\rho$ is a compatible coupling between $\mu_1$ and $\mu_2$, we have
\begin{eqnarray}
& &\div Q(x) =  \sum_{\gamma\,:\,\gamma^-=x} \rho(\gamma^-,\gamma^+)-\sum_{\gamma\,:\,\gamma^+=x}\rho(\gamma^-,\gamma^+)\\
& &= \sum_{y\,:\, x\leq y} \rho(x,y)-\sum_{y\,:\, y\leq x}\rho(y,x)=\mu_1(x)-\mu_2(x)\,.
\end{eqnarray}

\subsubsection{(2)$\Rightarrow$ (3)}\label{moskow}
Let $Q$ be a flow on a finite acyclic digraph $(V, \underline E)$ such that $\div Q=\mu_1-\mu_2$.
In order to generate a compatible coupling $\rho$ between $\mu_1$ and $\mu_2$, we use a variation of the algorithmic construction in \cite{PS} that is a discrete version of the original decomposition due to S.K. Smirnov on bounded domains of $\mathbb R^n$ \cite{SS} and  associates a coupling to a finite acyclic flow.\\
 Define $V_-:=\left\{x\in V\,:\,\mu_1(x) > \mu_2(x) \right\}$ and $V_+:=\left\{x\in V\,:\,\mu_2(x) > \mu_1(x) \right\}$.
First of all we show that it is possible to decompose the flow like
\begin{equation}\label{finite-decomp2}
Q=\sum_nq_n Q_{\gamma_n}\,
\end{equation}
where the paths $\gamma_n$ are such that
$\gamma_n^-\in V_-$ and $\gamma_n^+\in V_+$ for any $n$. Consider any finite decomposition of $Q$
and suppose that for example there exists a site $x\in V_-$ and a $m$
such that $\gamma_{m}^+=x$ (the other cases can be handled similarly).
Since by definition $\mu_1(x)> \mu_2(x)$ there exist necessarily some paths $\left\{\gamma_n\right\}_{n\in \mathcal N}$
of the decomposition such that $\gamma_n^-=x$ for any $n\in \mathcal N$ and moreover $\sum_{n\in \mathcal N}q_n>q_{m}$. We can then find
some weights $\left\{q'_n\right\}_{n\in \mathcal N}$ such that $\sum_{n\in \mathcal N}q'_n=q_{m}$ and $q'_n\leq q_n$.
With these weights we construct the new decomposition
\begin{equation}\label{bella-dec}
\sum_{n\in \mathcal N}\Big[q'_n Q_{\gamma_{m}\star \gamma_n}+ (q_n-q'_n)Q_{\gamma_n}\Big]+\sum_{n\not\in \mathcal N\cup m}q_nQ_{\gamma_n}\,.
\end{equation}
Since $(V, \underline E)$ is acyclic the paths obtained by concatenation are still self avoiding.
Performing a finite number of times a procedure of this type the final decomposition will have the required property.

Consider now a decomposition such that $\gamma_n^-\in V_-$ and $\gamma_n^+\in V_+$ for any $n$. This condition
immediately implies that
\begin{equation}
\sum_{\{n\,:\, \gamma_n^-=x\}} q_n= \mu_1(x)-\mu_2(x)\,, \qquad x\in V_-\,.
\label{sconfitta}
\end{equation}
We get
\begin{equation}
\sum_{n}q_n = \sum_{x\in  V_-}\sum_{\{n\,:\, \gamma_n^-=x\}} q_n = \sum_{x\in  V_-} \big[\mu_1(x)-\mu_2(x)\big]
=\frac 12 \sum_x\left|\mu_1(x)-\mu_2(x)\right|\,.
\label{cambio}
\end{equation}
In particular we deduce that the l.h.s. of \eqref{cambio} is smaller or equal than 1.

Using this special decomposition
we can construct the coupling. We define
\begin{equation}
\rho(x,y):=\left\{
\begin{array}{ll}
\min\left\{\mu_1(x),\mu_2(x)\right\} &  \textrm{if}\ x=y\,,\\
\sum_{\left\{n\,:\, \gamma_n^-=x\,, \gamma_n^+=y\right\}}q_n & \textrm{if}\ x\neq y\,.
\end{array}
\right.
\label{funziona}
\end{equation}
Using \eqref{sconfitta} and the analogous formula for $V_+$ it is easy
to verify that $\rho$ defined in \eqref{funziona} is a coupling between
$\mu_1$ and $\mu_2$. This coupling is clearly compatible
since if there exists a $\gamma_n$ such that $\gamma_n^-=x$ and $\gamma_n^+=y$
then necessarily $x\leq y$. This completes the proof.\\

Note that for the coupling given above we have
\begin{equation}\label{dopofunziona}
\sum_{x\neq y}\rho(x,y)=\sum_nq_n=\frac 12\sum_x\left|\mu_1(x)-\mu_2(x)\right|\,,
\end{equation}
so that any coupling constructed in this way is an optimal one.

\subsection{Proof of Theorem \ref{ilteorema-infinito}}
\label{Strassen-infinito}

\subsubsection{(2) $\Rightarrow$ (1)}
Let $Q$ be a finitely decomposable flow such that $\div Q=\mu_1-\mu_2$. Then, recalling
\eqref{divQ-fd} and using the summability of the weights $q_n$, we have for any increasing function $f\in L^\infty(V)$
\begin{equation*}
\mu_2(f)-\mu_1(f)=-\sum_{x\in V}f(x)\div Q(x)=\sum_{n}q_n\left(f\left(\gamma_n^+\right)-f\left(\gamma_n^-\right)\right)\geq 0\,.
\end{equation*}

\subsubsection{(1) $\Rightarrow$ (2)}

We start with a preliminary result.
Let $D\subseteq L^1(V)$ be the subset of functions that can be obtained
as divergence of a finitely decomposable flow. The subset $D$ is clearly convex
and we have the following result.
\begin{lemma}
The subset $D$ is closed in $L^1(V)$.
\end{lemma}
\begin{proof}
Let  $\{f^{(n)}\}_n \subset D$ be a sequence which converges to $f\in L^1(V)$. Since $f^{(n)}\in D$ then there exists a sequence of
finitely decomposable flows $Q^{(n)}$ such that $\div Q^{(n)}=f^{(n)}$. We need to show that there exists
a finitely decomposable flow $Q$ such that $\div Q=f$.
First of all by \eqref{divQ-fd} it follows that $\sum_xf^{(n)}(x)=0$, then the  measures $\mu^{(n)}_1=f^{(n)}\vee 0$ and $\mu^{(n)}_2=-\left(f^{(n)}\wedge 0\right)$ have the same (finite) total mass.
We can then apply the algorithmic construction of Section \ref{Strassentot} to the flow $Q^{(n)}$ obtaining as a result a compatible coupling $\rho^{(n)}$ between the measures $\mu^{(n)}_1$ and $\mu^{(n)}_2$. Now, for any pair of different vertices $x\leq y$ we fix one path $\gamma_{x,y}$ going from $x$ to $y$ and we define the flow $\tilde Q^{(n)}:=\sum_{\left\{ x\neq y\right\}}\rho^{(n)}(x,y)Q_{\gamma_{x,y}}$.
By construction we have $\div \tilde Q^{(n)}=f^{(n)}$.
By \eqref{fini-ind-mea}, the sequence of flows $\left\{\tilde Q^{(n)}\right\}_{n\in \mathbb N}$ induces a sequence
$\left\{M^{(n)}\right\}_{n\in \mathbb N}$ of finite measures on the set $\Gamma$ of finite self avoiding paths.
Let us show that this sequence is tight.
Relation \eqref{dopofunziona} holds also in the infinite as can be easily proved using the limit argument.
Then we have
\begin{equation}\label{ascule2}
\sum_\gamma M^{(n)}(\gamma)=\sum_{\left\{x\leq y\, x\neq y\right\}}\rho^{(n)}(x,y)=\frac 12 \sum_x\left|f^{(n)}(x)\right|\,.
\end{equation}
Since $\{f^{(n)}\}$ converges to $f$ in $L^1(V)$, we have that the r.h.s. of \eqref{ascule2} converges to $\frac 12 \sum_x|f(x)|<+\infty$
and this implies that the l.h.s. of \eqref{ascule2} is uniformly bounded.\\
Now, let $V_n$ be an invading sequence of vertices and define
$$\tilde V_n:=\{z\,:\, \exists x,y\in V_n \ \textrm{with}\ z\in \gamma_{x,y}\}\,.$$
We define also $\Gamma_n:=\left\{\gamma\in \Gamma\,:\, \gamma \subseteq \tilde V_n\right\}$.
We have
\begin{eqnarray}
& &M^{(n)}\left(\Gamma_k^c\right)=\sum_{\left\{\gamma\,:\,\left\{\gamma^-,\gamma^+\right\}\cap V_k^c\neq \emptyset\right\}}M^{(n)}(\gamma)\nonumber \\
& &\leq \sum_{x\not\in V_k}\sum_y\rho^{(n)}(x,y)+\sum_{y\not\in V_k}\sum_x\rho^{(n)}(x,y)= \sum_{x\not \in V_k}\left|f^{(n)}(x)\right|\,.
\label{up-up}
\end{eqnarray}
Tightness follows now directly from \eqref{up-up}, the convergence of
$f^{(n)}$ to $f$ and the summability of $f$.
By Prohorov Theorem, the sequence is relatively compact.
Let $M=\sum_kq_k\delta_{\gamma_k}$ be the weak limit of a subsequence of $\{M^{(n)}\}$ and let  us consider the finitely decomposable flow $Q:=\sum_k q_kQ_{\gamma_k}$.
Then, at least along a subsequence $\{n^\prime\}$, using the weak
convergence of $\{M^{(n^\prime)}\}$ to $M$ and the (point-wise) convergence
of $f^{(n^\prime)}$ to $f$, we have
\begin{eqnarray*}
\div Q(x)&=&\sum_{k\,:\, \gamma_k^-=x}q_k-\sum_{k\,,:\, \gamma_k^+=x}q_k\\
& =&\lim_{n^\prime \to +\infty}\left(\sum_{y}\rho^{(n^\prime )}(x,y)-
\sum_{y}\rho^{(n^\prime )}(y,x)\right)\\
&=&\lim_{n^\prime \to +\infty}f^{(n^\prime )}(x)=f(x)\,.
\end{eqnarray*}
This means that $Q$ is a finite decomposable flow such that $\div Q=f$
and consequently $f\in D$.
\end{proof}
Let us suppose that $\mu_2(f)-\mu_1(f)\geq 0$ for any increasing function
$f\in L^\infty(V)$. This can be written as $\langle f,\mu_2-\mu_1\rangle_V\geq 0$ where
$\mu_2-\mu_1\in L^1(V)$ and $\langle\cdot ,\cdot \rangle_V$ is the $L^\infty(V)$, $L^1(V)$
dual paring defined by
\begin{equation}\label{dual}
\langle f,g\rangle_V:=\sum_{x\in V}f(x)g(x)\,, \qquad f\in L^\infty(V)\,, g\in L^1(V)\,.
\end{equation}
Let $Q_e$ for $e\in E$ be the flow defined by \eqref{flowcycl} for the elementary path $\gamma$
given by the single edge $e$. A function $f\in L^\infty(V)$ is increasing if and only if $\nabla f\in L^\infty(E)$, defined for $(x,y)\in E$ by $\nabla f (x,y)=f(y)-f(x)$,
is such that
\begin{equation}\label{incr-dual}
\langle \nabla f,Q_e\rangle_E\geq 0\,, \qquad \forall e\in E\,,
\end{equation}
where $\langle \cdot,\cdot \rangle_E$ is the dual pairing for functions on edges.

We need to show that $\mu_1-\mu_2\in D$.
Let us suppose by contradiction that this is not the case. Since $D$ is convex and
closed, by Hahn-Banach Theorem we deduce that
there exists an $f^*\in L^\infty (V)$ such that
\begin{equation}\label{hahn}
\left\{
\begin{array}{l}
\langle f^*, \div Q\rangle_V<0\,, \qquad \forall\  Q\ \textrm{finitely\ decomposable}\,,\\
\langle f^*, \mu_1-\mu_2\rangle_V>0\,.
\end{array}
\right.
\end{equation}
In particular by the first inequality we have that
$$
-\langle f^*, \div Q_e\rangle_V=\langle \nabla f^*,Q_e\rangle_E>0\,, \qquad \forall e\in E\,,
$$
which means that $f^*$ is increasing. This fact together with the second inequality in
\eqref{hahn} gives a contradiction.

\subsubsection{(3) $\Rightarrow$ (2)}
The proof of this implication is the same given for the finite case.\\

\subsubsection{(2) $\Rightarrow$ (3)}
We need to extend the construction given in the proof of Theorem \ref{ilteorema} to the infinite case.
Let $Q$ be a
finitely decomposable flow such that $\div Q=\mu_1-\mu_2$.  Let $V_n$ be an invading
sequence of vertices such that $\cup_{k\leq n}\gamma_k \subseteq V_n$ where the paths $\gamma_k$
are the ones involved in the finite decomposition \eqref{finite-decomp} of $Q$.

For each $n$ we consider the finite digraph having vertices $V_n\cup \{g\}$ where
$g$ is a ghost site. The set of edges $E_n$ contains all edges $(x,y)\in E$ such that
$x,y\in V_n$, moreover it contains the edges $(g,z)$ or $(z,g)$ with $z\in V_n$ if respectively
there exists an $(x,z)\in E$ such that
$x\not\in V_n$ or there exists an $(z,y)\in E$ such that
$y\not\in V_n$. The digraph $(V_n\cup \{ g\} ,E_n)$ is not necessarily acyclic.
Starting from the flow $Q$ on $(V,E)$ we construct a flow $Q_n$
on $(V_n\cup \{g\} ,E_n)$ as follows
\begin{equation}
Q_n(x,y):=\left\{
\begin{array}{ll}
Q(x,y) & \textrm{if} \ x,y\in V_n\,, \\
\sum_{z\not\in V_n}Q(x,z) & \textrm{if}\ y=g\,, x\in V_n \\
\sum_{z\not\in V_n}Q(z,y) & \textrm{if}\ x=g\,, y\in V_n\,.
\end{array}
\right.
\label{messa}
\end{equation}
The series appearing in \eqref{messa} are convergent since $Q$ has a well defined divergence.
We have
$$
\div Q_n(x)=\left\{
\begin{array}{ll}
\div Q(x) =\mu_1(x)-\mu_2(x)\,,& \textrm{if}\ x\in V_n\,,\\
\sum_{y\not\in V_n} \big(\mu_1(y)-\mu_2(y)\big)\,, & \textrm{if}\ x=g\,.\\
\end{array}
\right.
$$
In general the flow $Q_n$ will not be acyclic, but removing cycles we can obtain an acyclic flow $Q^*_n$ having
the same divergence as $Q_n$.
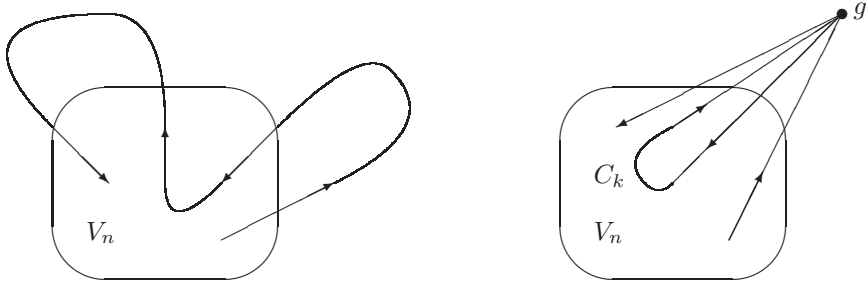
\begin{figure}\label{erasure}
\setlength{\unitlength}{1.5cm}
\begin{picture}(4.5,3)

\put(0, 1){\oval(2, 1.7)}
\put(4.5, 1){\oval(2, 1.7)}
\put(-0.7,0.5){$V_n$}
\put(3.8,0.5){$V_n$}
\put(3.8,1.0){$C_k$}

\put(6,2.5){\circle*{0.1}}
\put(6.1,2.5){$g$}
\put(5,0.5){\line(1,2){1}}
\put(5,0.5){\vector(1,2){0.3}}
\put(6,2.5){\vector(-2,-1){2}}



\put(0.5,0.5){\vector(2,1){1}}
\put(1,1.5){\vector(-1,-1){0.5}}
\put(0,1){\vector(0,1){0.5}}
\put(-1,1.5){\vector(1,-1){0.5}}
\qbezier(1.5,1)(2.5,1.5)(2,2)
\qbezier(1,1.5)(1.75,2.25)(2,2)
\qbezier(0.5,1)(0,0.5)(0,1)
\qbezier(0,1.5)(0,2.5)(-0.5,2.5)
\qbezier(-1.0,1.5)(-2.0,2.5)(-0.5,2.5)


\put(4.5,1.5){\line(3,2){1.5}}
\put(4.5,1.5){\vector(3,2){0.3}}
\put(6,2.5){\line(-1,-1){1.5}}
\put(6,2.5){\vector(-1,-1){1.2}}

\qbezier(4.5,1)(4.375,0.875)(4.25,1)
\qbezier(4.5,1.5)(4,1.25)(4.25,1)

\end{picture}
\caption{A path $\gamma_k$ entering and exiting several times in and from $V_n$ (left). The same path
after the identification of the sites outside $V_n$ with the ghost site $g$ (right). The cycle $C_k$ containing $g$
has to be removed.}
\end{figure}
This is done in two steps.
The first one is as follows. Consider a path $\gamma_k$ of the finite decomposition of $Q$ such that it exits from and enters in $V_n$ several times. This is possible only if $k>n$. After the identification of all the sites
outside $V_n$ with the single ghost site $g$ the path will not be anymore self--avoiding (see Figure $1$). If we remove the cycles that have been created (all of which will contain the ghost site) the self--avoiding path that we obtain will exit from $V_n$ or enter in $V_n$ at most once. The corresponding transformation on the flow $Q_n$ is the following. Let us consider the example of Figure $1$ and call $C_k$ the cycle in $(V_n\cup \{g\} ,E_n)$ created after the identification of all the sites
outside $V_n$ with the single ghost site $g$. By construction we have $q_kQ_{C_k}\subseteq Q_n$ so that
$Q_n-q_kQ_{C_k}$ is still a flow on $(V_n\cup \{g\} ,E_n)$ having the same divergence as $Q_n$.
We consider iteratively each path $\gamma_k$ of the original cyclic
decomposition with $k>n$ and remove the cycles as illustrated above.
After this deletion procedure
the flow $Q_n^1$ obtained is still not necessarily acyclic.

The second step is as follows. Since the digraph $(V_n\cup \{g\} ,E_n)$ is finite we can consider a finite number of cycles $C'_i$ and weights $q'_i$ such that $Q_n^*:=Q_n^1-\sum_iq'_iQ_{C_i'}$ is acyclic (see for example the construction in \cite{Io}).
The choice of the cycles and weights in this last step is arbitrary.

To the values of the flow
$Q^*_n$ on edges entering or exiting from the ghost site $g$ can contribute only
the paths $\gamma_k$ with $k>n$. Moreover the deletion procedure outlined above guarantees
that every single path $\gamma_k$ with $k>n$ may contribute
no more than once to the total flux entering in $g$ (that is $\sum_{x\in V_n}Q^*_n(x,g)$) or to the total
flux exiting from $g$ (that is $\sum_{x\in V_n}Q^*_n(g,x)$).
This means that we have the bounds
\begin{equation}\label{fl-bon}
\left\{
\begin{array}{l}
\sum_{x\in V_n}Q_n^*(x,g)   \leq \sum_{k>n} q_k\,, \\
\sum_{x\in V_n}Q_n^*(g,x)   \leq \sum_{k>n} q_k\,.
\end{array}
\right.
\end{equation}
Let us now consider the flow $\tilde Q^*_n$ such that $\tilde Q^*_n(x,y)=Q^*_n(x,y)$
when both $x$ and $y$ belong to $V_n$ and $\tilde Q^*_n(x,y)=0$ otherwise.
The flow $\tilde Q^*_n$ can be naturally interpreted as a flow on the original
digraph $(V,E)$ and by construction $\tilde Q^*_n\leq Q$.
We have also
\begin{equation}\label{san-ben}
\div \tilde Q^*_n(x) =\mu_1(x)-\mu_2(x)+\delta_n(x)\,, \qquad x\in V_n\,,
\end{equation}
where by \eqref{fl-bon} we have
\begin{equation}\label{oravero}
\sum_x |\delta_n(x)|\leq 2 \sum_{k> n} q_k\,.
\end{equation}
We define the following sequences of positive measures on $V$
\begin{equation}
\label{aer}
\mu_1^{(n)}(x):=\left\{
\begin{array}{ll}
 \mu_1(x)+\delta_n(x) & \textrm{if}\ x\in V_n\,, \delta_n(x)>0\,, \\
\mu_1(x) & \textrm{if}\ x\in V_n\,, \delta_n(x)\leq 0\,, \\
0 & \textrm{if}\ x\not \in V_n\,,
\end{array}
\right.
\end{equation}
\begin{equation}
\label{eo}
\mu_2^{(n)}(x):=\left\{
\begin{array}{ll}
 \mu_2(x)-\delta_n(x) & \textrm{if}\ x\in V_n\,, \delta_n(x)<0\,, \\
\mu_2(x) & \textrm{if}\ x\in V_n\,, \delta_n(x)\geq 0\,, \\
0 & \textrm{if}\ x\not \in V_n\,.
\end{array}
\right.
\end{equation}
We have $\div \tilde Q^*_n=\mu^{(n)}_1- \mu^{(n)}_2$ and $\sum_x\left( \mu^{(n)}_1(x)- \mu^{(n)}_2(x)\right)=0$.
Since $\left|E\left(\tilde Q_n^*\right)\right|<+\infty$ we can apply the finite algorithmic construction used in the proof of Theorem \ref{ilteorema} (which works also for pairs of finite positive measures having the same total mass) obtaining from the acyclic flow $\tilde Q^*_n$ a measure
$\rho^{(n)}$ on $V\times V$ such that $\sum_x  \rho^{(n)}(x,y)=\mu_2^{(n)}(y)$
and $\sum_y  \rho^{(n)}(x,y)=\mu_1^{(n)}(x)$. Since the mass is transported
along edges of the original digraph $(V,E)$ we deduce that $\rho^{(n)}(x,y)=0$ if $x\not\leq y$.
We show that the sequence of positive measures
$\rho^{(n)}$ has total mass uniformly bounded and is tight.
The bound on the mass follows by
$$
\sum_{x}\sum_y\rho^{(n)}(x,y)=\sum_x\mu_1^{(n)}(x)\leq\sum_x\left[\mu_1(x)+
|\delta_n(x)|\right]\leq 1+2\sum_{k=1}^{+\infty} q_k\,.
$$
The tightness follows by the following argument.
Fix an arbitrary $\epsilon >0$ and let $m^*$ be an integer number such that
$$
\max\left\{\mu_1\left(V_{m^*}^c\right), \mu_2\left(V_{m^*}^c\right) \right\}<\epsilon\,,
$$
where the upper index $c$ denotes the complementary set.
Fix also $n^*$ such that $2\sum_{k=n^*}^{+\infty}q_k<\epsilon$.
Then we have for any $n>n^*$ and $m>m^*$
\begin{eqnarray*}
& &\rho^{(n)}\Big(\left(V_m\times V_m\right)^c\Big)=\rho^{(n)}\Big(V_m^c\times V_m\Big)+
\rho^{(n)}\Big(V_m\times V_m^c\Big)+\rho^{(n)}\Big(V_m^c\times V_m^c\Big)\\
&\leq&\rho^{(n)}\Big(V\times V_m^c\Big)+
\rho^{(n)}\Big(V_m^c\times V\Big)=\mu^{(n)}_2(V_m^c)+\mu^{(n)}_1(V_m^c)\\
& \leq & \mu_2(V_m^c)+\mu_1(V_m^c)+\sum_x|\delta_n(x)|\leq 3\epsilon\,.
\end{eqnarray*}

By Prokhorov Theorem there exists a subsequence, that we still call $\rho^{(n)}$, that is weakly convergent.
Let us call $\rho$ its weak limit. Since by \eqref{oravero}, \eqref{aer} and \eqref{eo} $\mu_i^{(n)}(x)\to \mu_i(x)$ for any $x$ we immediately
obtain
$$
\sum_y\rho(x,y)=\lim_{n\to +\infty}\sum_y\rho^{(n)}(x,y)=\lim_{n\to +\infty} \mu_1^{(n)}(x)=\mu_1(x)\,.
$$
A similar result holds for $\mu_2$.
This means that $\rho$ is a coupling between $\mu_1$ and $\mu_2$. Since $\rho^{(n)}(x,y)=0$ when $x\not\leq y$ this will be true
also for the limiting measure $\rho$. This completes the proof.\\


\section{Proof of Propositions \ref{proposizione} and  \ref{proposizione3}}

\label{PPP}
In this section we give the proofs of the auxiliary Propositions \ref{proposizione} and \ref{proposizione3} that are
useful to identify finitely decomposable flows.

\begin{proof}[Proof of Proposition \ref{proposizione}]
Consider the invading sequence $V_n$ for which  the outgoing flux towards infinity
$\sum_{x\in V_n,y\not \in V_n}Q(x,y)=:\phi^+_n$ is converging to zero when $n$ diverges.
Since
\begin{equation}
\sum_{x\in V_n,y\not \in V_n}Q(x,y)-\sum_{x\not\in V_n,y \in V_n}Q(x,y)= \mu_1(V_n)-\mu_2(V_n)\,,
\label{solo}
\end{equation}
then also the incoming flux from infinity $\sum_{x\not\in V_n,y \in V_n}Q(x,y):=\phi^-_n$ is converging to zero
when $n$ diverges. All the series in \eqref{solo} are convergent since $|V_n|<+\infty$ and the series appearing in the definition
of $\div Q$ \eqref{divQ} are supposed to be summable.

For each $n$ we consider the finite digraph having vertices $V_n\cup g_-\cup g_+$ where
$g_\pm$ are ghost sites. The set of edges $E_n$ contains all edges $(x,y)\in E$ such that
$x,y\in V_n$, moreover it contains edges of type $(g_-,z)$ or $(z,g_+)$ with $z\in V_n$ if respectively
there exists an $(x,z)\in E$ such that
$x\not\in V_n$  or there exists an $(z,y)\in E$ such that $y\not\in V_n$. Since the original graph is acyclic also this new finite digraph is acyclic.
Starting from the flow $Q$ on $(V,E)$ we associate to it a flow $Q_n$
on $(V_n\cup g_-\cup g_+ ,E_n)$ as follows
$$
Q_n(x,y):=\left\{
\begin{array}{ll}
Q(x,y) & \textrm{if} \ x,y\in V_n\,, \\
\sum_{z\not\in V_n}Q(x,z) & \textrm{if}\ y=g_+\,, x\in V_n \\
\sum_{z\not\in V_n}Q(z,y) & \textrm{if}\ x=g_-\,, y\in V_n\,.
\end{array}
\right.
$$
We have
$$
\div Q_n(x) =\div Q(x)=\mu_1(x)-\mu_2(x)\,, \qquad  \ x\in V_n\,.
$$
We have also $\div Q_n(g_-)=\phi_n^-$ and $\div Q_n(g_+)=-\phi_n^+$. Let us introduce two
sequences of positive measures on $V_n\cup g_-\cup g_+$ defined as
\begin{equation}\label{sissi}
\mu_1^{(n)}(x):=\left\{
\begin{array}{ll}
\mu_1(x)& \textrm{if}\ x\in V_n\,, \\
\phi_n^- & \textrm{if}\ x=g_-\\
0 & \textrm{if}\ x=g_+\,,
\end{array}
\right.
\end{equation}
\begin{equation}\label{sissi1}
\mu_2^{(n)}(x):=\left\{
\begin{array}{ll}
\mu_2(x)& \textrm{if}\ x\in V_n\,, \\
\phi_n^+ & \textrm{if}\ x=g_+\\
0 & \textrm{if}\ x=g_-\,.
\end{array}
\right.
\end{equation}
Then $Q_n$ is a flow on a finite acyclic digraph and moreover $\div Q_n=\mu_1^{(n)}-\mu_2^{(n)}$. Applying the finite algorithmic construction we obtain a finite decomposition
\begin{equation}
Q_n=\sum_mq^{(n)}_mQ_{\gamma^{(n)}_m}
\label{non-so}
\end{equation}
for suitable weights $q^{(n)}_m$ and paths $\gamma^{(n)}_m$. The paths $\gamma^{(n)}_m$
are self--avoiding paths on the digraph $\left( V_n\cup g_-\cup g_+, E_n\right)$ but to every path $\gamma$ on this digraph
it can be easily associated a self--avoiding path $\tilde \gamma$ on the original digraph $(V,E)$. This is done simply
transforming any edge $(g_-,x)\in \gamma$ into an arbitrary edge $(y,x)\in E(Q)$ with $y\not\in V_n$ and
any edge $(x, g_+)\in \gamma$ into an arbitrary edge $(x,y)\in E(Q)$ with $y\not\in V_n$. After this identification
we obtain an acyclic finitely decomposable flow on $(V,E)$
\begin{equation}
\tilde Q_n:=\sum_m q^{(n)}_mQ_{\tilde \gamma^{(n)}_m}\,.
\label{daipra}
\end{equation}
By construction we have
\begin{equation}
\tilde Q_n(x,y)=Q(x,y)\,,
\label{ny}
\end{equation}
for any $n$ big enough so that $x,y\in V_n$.

Recall that $\Gamma$ is the countable set of all finite self-avoiding paths on the digraph
$(V,E)$. Let also $\Gamma_n\subseteq \Gamma$ be the subset of all the paths $\gamma\subseteq V_n$.
To the decomposition \eqref{daipra} we associate by \eqref{fini-ind-mea} a positive and finite measure
on $\Gamma$ given by
\begin{equation}
M^{(n)}:=\sum_m q^{(n)}_m\delta_{\tilde \gamma^{(n)}_m}\,.
\label{seq-gamma}
\end{equation}
The sequence of measures  $\left\{M^{(n)}\right\}_{n\in \mathbb N}$ is a sequence
of positive and finite measures on $\Gamma$.
We now show that this sequence of measures is tight and has total mass uniformly bounded. Recall that the
coefficients $q^{(n)}_m$ in \eqref{seq-gamma} are the same of \eqref{non-so}
so that by \eqref{dopofunziona} they satisfy
\begin{equation*}
\sum_{\gamma\in \Gamma} M^{(n)}(\gamma)=\sum_m q^{(n)}_m=\frac 12\Big[\phi^+_n+\phi_n^-+\sum_{x\in V_n}\left|\mu_1(x)-\mu_2(x)\right|\Big]\,.
\end{equation*}
Since $\phi_n^\pm$ are converging to zero we have an uniform bound on the total mass.
This guarantees that the total mass of $\left\{M^{(n)}\right\}_{n\in \mathbb N}$ is uniformly
bounded.
Moreover we have
\begin{equation}
M^{(n)}\left(\Gamma_k^c\right)= \sum_{\left\{m\,:\, \tilde\gamma^{(n)}_m\subseteq V_k^c\right\}}q_m^{(n)}
+\sum_{\left\{m\,:\, \tilde\gamma_m^{(n)}\cap V_k\neq \emptyset\,,\,
\tilde\gamma_m^{(n)}\cap V_k^c\neq \emptyset \right\}}q_m^{(n)}\,.
\label{tait-paths}
\end{equation}
The first term on the right hand side of \eqref{tait-paths} is $0$ when $n\leq k$.
When $n>k$ can be estimated using \eqref{sconfitta} as
\begin{eqnarray}
& &\sum_{\left\{m\,:\, \tilde\gamma^{(n)}_m\subseteq V_k^c\right\}}q_m^{(n)}\leq
\sum_{x\not \in V_k}\sum_{\left\{m\,:\, \tilde\gamma^{(n)-}_m=x\right\}}q_m^{(n)}\nonumber \\
& &\leq\sum_{x\not\in V_k}\left|\mu_1^{(n)}(x)-\mu_2^{(n)}(x)\right|\leq \phi_n^++\phi_n^-+\sum_{x\not\in V_k}\left|\mu_1(x)-\mu_2(x)\right|\,.
\label{primastima}
\end{eqnarray}
The second term in \eqref{tait-paths} can be directly estimated by $\phi^+_k+\phi^-_k$ independently of $n$. With these bounds
the tightness of the sequence of measures $\left\{M^{(n)}\right\}_{n\in \mathbb N}$ can be easily established using the condition of zero flux towards infinity.
By Prokhorov Theorem for positive finite measures, we can then extract a weak converging subsequence that
 we still call $M^{(n)}$ and call $M:=\sum_m q_m\delta_{\tilde \gamma_m}$
its limit. This is a finite and positive measure on $\Gamma$. The function that associate to any
path $\gamma$ the value $1$ if $(x,y)\in \gamma$ and zero otherwise is continuous and bounded on $\Gamma$
endowed of the discrete topology. By \eqref{ny}
we deduce that if we construct the flow $\tilde Q:=\sum_mq_mQ_{\tilde\gamma_m}$  then we have that
$\tilde Q(x,y)=\lim_{n\to +\infty} \tilde Q_n(x,y)=Q(x,y)$. This means that $\tilde Q$ coincides with the original flow $Q$. Since $\tilde Q$ is clearly finitely decomposable we are done.
\end{proof}

\begin{proof}[Proof of Proposition \ref{proposizione3}]
Let us suppose by contradiction that $\sum_n q_nQ_{\gamma_n}$ is a finite decomposition of
$Q$ and that \eqref{london} does not converge to zero
for an invading sequence. This means that there exists an $\epsilon$ and an infinite
sequence of edges $\left\{e_i\right\}_{i\in \mathbb N}$
such that $Q(e_i)>\epsilon$ for any $i$.
Let $n^*$ be such that $\sum_{n>n^*}q_n<\epsilon$. Let $e_{i^*}$ such that
$e_{i^*}\not\in \cup_{n\leq n^*}\gamma_n$. Then we have
$$
\epsilon <Q(e_{i^*})=\sum_{n>n^*}q_nQ_{\gamma_n}(e_{i^*})<\epsilon\,,
$$
a contradiction.
\end{proof}

\section{Examples}
\label{appl}
In this section we discuss some examples of applications of Theorems \ref{ilteorema}, \ref{ilteorema-infinito}. Even is simple they are conceptually important since we use arguments that are different from the usual ones. In example 5.1 we obtain the classic condition for stochastic monotonicity on $\mathbb Z$. Instead of construct a coupling we need just to perform a discrete integration. The same happens in example 5.2. In example 5.3 we show that the problem of stochastic monotonicity has a dual problem coinciding with the non emptiness of a polyhedron in a space whose dimension is the number of independent cycles of the Hasse diagram. A discrete Poisson equation can be relevant in this dual problem. Example 5.4 is a special issue of 5.3.
In example 5.5 we generalize a classic construction. A coupling of two random variables can be constructed writing them as functions of a common random variable. If the functions satisfy a monotonicity property then the coupling is a monotone one. We show that there is a similar construction for flows that works under less restrictive conditions on the functions.
In Example 5.6 we show that there is a natural construction for flows that is the counterpart of the product coupling. The mechanism is a bit tricky and works due to the presence of a telescopic sum.
In Example 5.7 we construct some special posets where the approach based on flows applies easily giving a necessary and sufficient condition for stochastic monotonicity. In Example 5.8 we obtain very shortly the result of \cite{Ho}.

\subsection{The one dimensional case}\label{1-d}
We discuss the simplest countable poset, that is $\mathbb Z$ with the usual partial order relation. We want to get the
well known \cite{LL} necessary and sufficient conditions to have $\mu_1\preceq \mu_2$, using item $(2)$ of Theorem \ref{ilteorema-infinito}.
In this case the partial
order can be described by the Hasse diagram corresponding to the acyclic digraph $(\mathbb Z, E)$
where $E=\left\{(x,x+1)\right\}_{x\in \mathbb Z}$.
The condition $\div Q=\mu_1-\mu_2$ reads
$$
Q(x,x+1)-Q(x-1,x)=\mu_1(x)-\mu_2(x)\,,
$$
and with a finite telescopic sum for any $y<x$ we get
\begin{equation}\label{1-d-flux}
Q(x,x+1)-Q(y,y+1)=\sum_{z=y+1}^x\big(\mu_1(z)-\mu_2(z)\big)\,.
\end{equation}
By Proposition \ref{proposizione3} a necessary condition to have that $Q$ is finitely decomposable
is that $\lim_{y\to -\infty}Q(y,y+1)=0$.
Taking the limit
$y\to-\infty$ in \eqref{1-d-flux} we then get
\begin{equation}\label{1-d-limit}
    Q(x,x+1)=\sum_{z=-\infty}^x\left(\mu_1(z)-\mu_2(z)\right)\,.
\end{equation}
This means that there is at most one finitely decomposable flow having divergence equal to $\mu_1-\mu_2$
that is \eqref{1-d-limit}. Consider the invading sequence $V_n:=\left\{-n,\dots ,n\right\}$. The flux exiting
from $V_n$ coincides with $Q(n,n+1)$ that by \eqref{1-d-limit} is converging to zero when $n\to +\infty$. By Proposition
\ref{proposizione} $Q$ is finitely decomposable.
The last condition that $Q$ has to satisfy to be a flow is $Q(x,x+1)\geq 0$ for any $x\in \mathbb Z$. This condition
reads
\begin{equation}\label{1-d-final}
    \sum_{z=-\infty}^x\left(\mu_1(z)-\mu_2(z)\right)=F_1(x)-F_2(x)\geq 0 \,,\qquad \forall x\in \mathbb Z\,,
\end{equation}
where $F_i(x):=\sum_{z=-\infty}^x\mu_i(z)$ is the distribution function
of the measure $\mu_i$.

\subsection{Finite and infinite trees}\label{fait}
We consider the case of posets described by digraphs $(V,E)$ such that
the associated graph $(V,\mathcal E)$ is a tree. We discuss both the finite and the infinite case.

Let us start with the finite case.
Removing one edge of $\mathcal E$
the graph is divided into two connected components.
If the edge that has been removed is $\{x,y\}$ and $x\leq y$ we call $T_-^{\{x,y\}}$ the connected component
containing $x$ and $T_+^{\{x,y\}}$ the connected component containing $y$.
Using a discrete Gauss Green identity we get that there is a unique solution
to the equation $\div Q=\mu_1-\mu_2$ that is
\begin{equation}
Q(x,y)=\sum_{z\in T_-^{\{x,y\}}}\big(\mu_1(z)-\mu_2(z)\big)\,.
\label{domani}
\end{equation}
The left hand side of \eqref{domani} is the flux from $T_-^{\{x,y\}}$ to $T_+^{\{x,y\}}$ while
the right hand side is the sum of the divergences in $T_-^{\{x,y\}}$.
Since $Q$ has to be a flow on $(V,E)$ it must be positive and this gives
\begin{equation}
\sum_{z\in T_-^e}\left(\mu_1(z)-\mu_2(z)\right)\geq0\,, \qquad \forall e\in \mathcal E\,
\label{non}
\end{equation}
that is the necessary and sufficient condition to have $\mu_1\preceq \mu_2$.

If $(V,\mathcal E)$ is an infinite tree then the equation $\div Q=\mu_1-\mu_2$ has not
an unique solution. However, if $Q$ is a finitely decomposable flow we have
\begin{eqnarray}
Q(x,y)&=&\sum_{\{n\,:\, \gamma_n^-\in T_-^{\{x,y\}}\}}q_n-\sum_{\{n\,:\, \gamma_n^+\in T_-{\{x,y\}}\}} q_n\nonumber \\
&=&\sum_{\{z\in T_-^{\{x,y\}}\}}\div Q(z)=\sum_{\{z\in T_-^{\{x,y\}}\}}\left(\mu_1(z)-\mu_2(z)\right)\,.
\label{prima-mezza}
\end{eqnarray}
This means that there is at most one finitely
decomposable solution to the equation $\div Q=\mu_1-\mu_2$ that is still given by \eqref{domani}. Indeed, as in subsection
 \ref{1-d}, using Proposition \ref{proposizione} it can be easily shown that this solution is finitely decomposable. The positivity $Q$ gives the same condition \eqref{non} of the finite case.

\subsection{A dual problem}\label{egi}

We consider the case $|V|<+\infty$.
 A discrete vector field on $(V,\mathcal{E})$ is a  map $\phi$ on pairs of ordered vertices
$(x,y)$ with $\{x,y\}\in \mathcal{E}$ satisfying the condition $\phi(x,y)=-\phi(y,x)$.\\
Let $\Lambda(\mathcal E)$ be the vector space of discrete vector fields on $(V,\mathcal E)$.
This is a $\left|\mathcal E\right|$
dimensional vector space. The following are classic results
(see \cite{Biggs} or \cite{Io} for a short introduction) and we give just a short informal overview.
A discrete vector field $\phi$ is a gradient if there exists a function
$f: V\to \mathbb R$ such that $\phi(x,y)=f(y)-f(x)$. The divergence at $x\in V$ of a discrete vector field $\phi$ is defined by
$\div \phi (x):= \sum_{y\,:\, \{x,y\}\in \mathcal E}\phi(x,y)$.
We have the orthogonal decomposition
$$
\Lambda(\mathcal E)=\Lambda_g(\mathcal E)\oplus \Lambda_{d}(\mathcal E)\,,
$$
where $\Lambda_g(\mathcal E)$ is the $|V|-1$ dimensional subspace of gradient discrete vector fields and
$\Lambda_d(\mathcal E)$ is the $|\mathcal E|-|V|+1$ dimensional subspace of divergence free discrete vector fields.
The orthogonality is with respect to the scalar product
$$\sum_{\{x,y\}\in \mathcal E}\phi(x,y)\psi(x,y)\,, \qquad \phi,\psi \in \Lambda(\mathcal E)\,.$$
A basis for $\Lambda_d(\mathcal E)$ is obtained choosing a suitable collection of divergence free discrete vector fields naturally associated to elementary independent cycles.
Fix $(V,\mathcal T)$ a spanning tree of $(V,\mathcal E)$, in particular $|\mathcal T|=|V|-1$.
For any $e\in \mathcal E\setminus \mathcal T$ the graph $(V,\mathcal T\cup \{e\})$ contains an unique cycle with distinct vertices.
Let us fix an arbitrary orientation on this cycle. On the graph $(V,\mathcal T\cup \{e\})$ there exists a unique, up to a multiplicative factor, divergence free discrete vector field $\phi_e$. This is defined by fixing $\phi_e(x,y)=1$ if $(x,y)$ belongs to the oriented cycle, $\phi_e(x,y)=-1$ if $(y,x)$ belongs to the oriented cycle and $\phi_e(x,y)=0$ otherwise. The collection $\left\{\phi_e\right\}_{e\in \mathcal E\setminus \mathcal T}$ is a basis of $\Lambda_d(\mathcal E)$.

All the discrete vector fields satisfying
\begin{equation}
\div \phi=\mu_1-\mu_2
\label{bb2}
\end{equation}
are given by
\begin{equation}
\phi^*+\sum_{e\in \mathcal E\setminus \mathcal T}\alpha_e\phi_e\,,
\label{cicl-dec}
\end{equation}
where the $\alpha_e$ are arbitrary real numbers and  $\phi^*$ is an arbitrary solution to \eqref{bb2},
for example of gradient type.

\smallskip

Let $(V,\underline E)$ be the Hasse diagram of a finite poset and
let $(V,\underline{\mathcal E})$ be the corresponding undirected graph.
Consider also $\mu_1$ and $\mu_2$ two probability measures on $V$.
The flows on $(V,\underline E)$ having divergence coinciding
with $\mu_1-\mu_2$ are in bijection with the discrete vector fields on $(V,\underline{\mathcal E})$ having the same divergence and
such that $\phi(x,y)\geq 0$ when $(x,y)\in \underline E$. The bijection is through the natural identification
$Q(x,y)=\phi(x,y)$ when $(x,y)\in \underline E$. The remaining values of the discrete vector
field are fixed by the antisymmetry condition.

Using the above construction we obtain that there
exists a flow on $(V,\underline E)$ having divergence $\mu_1-\mu_2$ if and only if the following conditions
are satisfied.
For edges $\{x,y\}\in \underline{\mathcal E}$ that do not belong to any cycle of the basis,
we have to impose $\phi^*(x,y)\geq 0$ when $(x,y)\in \underline E$. If we call $\underline E'$ the set of remaining edges, we have to impose that there exists a collection of real numbers
$\{\alpha_e\}_{e\in \underline{\mathcal E}\setminus \mathcal T}$
such that
\begin{equation}
\phi^*(x,y)+\sum_{e}\alpha_e\phi_e(x,y)\geq 0\,,
\qquad \forall (x,y)\in \underline E'\,.
\label{bbb}
\end{equation}
Recall that $\phi_e(x,y)$ is taking just the values $-1,0,+1$.
Conditions \eqref{bbb} in the $\alpha$ variables is equivalent to the statement
that a polyhedron on $\mathbb R^{|\underline E|-|V|+1}$ obtained
as the intersection of $|\underline E'|$ half-spaces (one for each $(x,y)\in \underline E'$)
is not empty. The interesting feature is that it is a geometric problem on a space
of dimension equal to the number of independent cycles of the Hasse diagram.

\smallskip

Consider for example the Hasse
diagram of Figure 2 (left) having one single cycle and such that $\underline E'=\underline E$.
Since the Hasse diagram has only one independent cycle the
stochastic monotonicity condition will reduce to a one dimensional problem.
Choosing arbitrarily one orientation we can label vertices as $V:=\left\{1,2,\dots ,n\right\}$  and
the edges as $\mathcal E:=\big\{\{x,x+1\}\big\}_{x=1}^n$ where the sum is modulo $n$.
Equation \eqref{cicl-dec} reduces to
\begin{equation}
\phi^\alpha:=\phi^*+\alpha\phi\,,
\label{salento}
\end{equation}
where $\phi(x,x+1)=1$ for any $x$, $\alpha$ is an arbitrary real number and $\phi^*$
is any given discrete vector field such that $\div \phi^*=\mu_1-\mu_2$. We can fix for example
\begin{equation}\label{fitoro}
\phi^*(x,x+1)=\sum_{y=1}^x \big(\mu_1(y)-\mu_2(y)\big)\,, \qquad x=1,\dots ,n\,.
\end{equation}
Let $\underline E^+:=\left\{(x,y)\in \underline E\,:\, y=x+1\right\}$ and
$\underline E^-$ the complementary set. Conditions \eqref{bbb} become
\begin{equation}
\left\{
\begin{array}{ll}
\phi^*(x,y)+\alpha \geq 0\,, & (x,y)\in \underline E^+\,, \\
\phi^*(x,y)-\alpha \geq 0\,, & (x,y)\in \underline E^-\,,
\end{array}
\right.
\end{equation}
that are equivalent to the single inequality
\begin{equation}\label{vecchio}
\max_{(x,y)\in \underline E^+}\left\{\sum_{z=1}^x \big(\mu_2(z)-\mu_1(z)\big)\right\}\leq \min_{(x,y)\in \underline E^-}\left\{\sum_{z=1}^y \big(\mu_2(z)-\mu_1(z)\big)\right\}\,.
\end{equation}
Condition \eqref{vecchio} is a necessary and sufficient condition to have $\mu_1\preceq \mu_2$ on a poset
like the one on the left of Figure 2. If we consider the special case on the right hand side of Figure 2 \eqref{vecchio} becomes
\begin{equation}
\big|\mu_1(B)-\mu_2(B)\big|+\big|\mu_1(C)-\mu_2(C)\big|\leq \big(\mu_1(A)-\mu_2(A)\big)-\big(\mu_1(D)-\mu_2(D)\big)\,.
\label{el-simple}
\end{equation}
It is not immediate to get the single inequality \eqref{el-simple} without condition $(2)$ of Theorem \ref{ilteorema}.

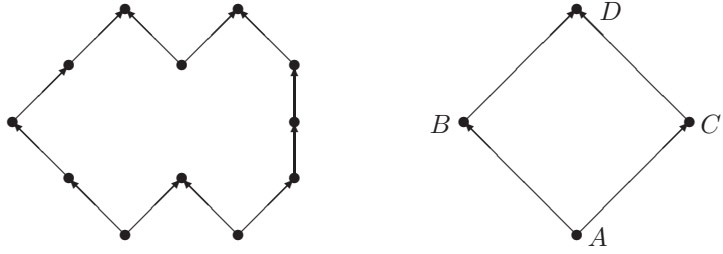
\begin{figure}\label{el-hasse}
\setlength{\unitlength}{1.5cm}
\begin{picture}(4,2.5)

\put(0,0){\circle*{0.1}}
\put(0.5,0.5){\circle*{0.1}}
\put(-0.5,0.5){\circle*{0.1}}
\put(1,0){\circle*{0.1}}
\put(1.5,0.5){\circle*{0.1}}
\put(1.5,1){\circle*{0.1}}
\put(1.5,1.5){\circle*{0.1}}
\put(-1,1){\circle*{0.1}}
\put(-0.5,1.5){\circle*{0.1}}
\put(0,2){\circle*{0.1}}
\put(0.5,1.5){\circle*{0.1}}
\put(1,2){\circle*{0.1}}

\put(0,0){\vector(-1,1){0.5}}
\put(0,0){\vector(1,1){0.5}}
\put(-0.5,0.5){\vector(-1,1){0.5}}
\put(1,0){\vector(-1,1){0.5}}
\put(1,0){\vector(1,1){0.5}}
\put(-1,1){\vector(1,1){0.5}}
\put(-0.5,1.5){\vector(1,1){0.5}}
\put(0.5,1.5){\vector(-1,1){0.5}}
\put(0.5,1.5){\vector(1,1){0.5}}
\put(1.5,1.5){\vector(-1,1){0.5}}
\put(1.5,0.5){\vector(0,1){0.5}}
\put(1.5,1){\vector(0,1){0.5}}

\put(4,0){\vector(-1,1){1}}
\put(4,0){\vector(1,1){1}}
\put(3,1){\vector(1,1){1}}
\put(5,1){\vector(-1,1){1}}
\put(4,0){\circle*{0.1}}
\put(4.1,-0.1){$A$}
\put(3,1){\circle*{0.1}}
\put(2.7,0.9){$B$}
\put(5,1){\circle*{0.1}}
\put(5.1,0.9){$C$}
\put(4,2){\circle*{0.1}}
\put(4.2,1.9){$D$}
\end{picture}
\caption{An Hasse diagram with one single cycle such that $\underline E'=\underline E$ (left).
The special case of the Hasse diagram of an elementary lattice (right).}
\end{figure}

\subsection{The two dimensional case}
Let us consider for simplicity two probability measures $\mu_1$ and $\mu_2$ on
the set $V=\mathbb{Z}^2 \cap \left([0,N]\times [0,M]\right)$. We denote by $\leq$ the usual partial
order relation on $\mathbb{Z}^2$, i.e., $(x_1,x_2)\leq (y_1,y_2)$ if $x_i\leq y_i$ for $i=1,2$.
Then $(V, \leq)$ is a poset with Hasse diagram as in Figure 3.
 We apply the general framework of section \ref{egi}. The number of independent cycles is $NM$ and in Figure 3 it is shown a choice of a basis of cycles one for each face of the squared lattice.
\begin{figure}
\setlength{\unitlength}{5cm}
\begin{picture}(1,1)

\multiput(0,0)(0.2,0){5}%
{\vector(0,1){0.2}}
\multiput(0,0.2)(0.2,0){5}%
{\vector(0,1){0.2}}
\multiput(0,0.4)(0.2,0){5}%
{\vector(0,1){0.2}}
\multiput(0,0.6)(0.2,0){5}%
{\vector(0,1){0.2}}
\multiput(0,0)(0.2,0){5}%
{\vector(1,0){0.2}}
\multiput(0,0.2)(0.2,0){5}%
{\vector(1,0){0.2}}
\multiput(0,0.4)(0.2,0){5}%
{\vector(1,0){0.2}}
\multiput(0,0.6)(0.2,0){5}%
{\vector(1,0){0.2}}
\multiput(0,0.8)(0.2,0){5}%
{\vector(1,0){0.2}}
\multiput(1,0)(0,0.2){4}%
{\vector(0,1){0.2}}

\put(0.1,0.1){\circle{0.14}}\put(0.17,0.11){\vector(0,1){0.00001}}
\put(0.3,0.3){\circle{0.14}}\put(0.37,0.31){\vector(0,1){0.00001}}
\put(0.5,0.3){\circle{0.14}}\put(0.57,0.31){\vector(0,1){0.00001}}
\put(0.3,0.1){\circle{0.14}}\put(0.37,0.11){\vector(0,1){0.00001}}
\put(0.5,0.1){\circle{0.14}}\put(0.57,0.11){\vector(0,1){0.00001}}
\put(0.7,0.1){\circle{0.14}}\put(0.77,0.11){\vector(0,1){0.00001}}
\put(0.9,0.1){\circle{0.14}}\put(0.97,0.11){\vector(0,1){0.00001}}
\put(0.1,0.3){\circle{0.14}}\put(0.17,0.31){\vector(0,1){0.00001}}
\put(0.7,0.3){\circle{0.14}}\put(0.77,0.31){\vector(0,1){0.00001}}
\put(0.9,0.3){\circle{0.14}}\put(0.97,0.31){\vector(0,1){0.00001}}

\put(0.1,0.5){\circle{0.14}}\put(0.17,0.51){\vector(0,1){0.00001}}
\put(0.7,0.5){\circle{0.14}}\put(0.77,0.51){\vector(0,1){0.00001}}
\put(0.9,0.5){\circle{0.14}}\put(0.97,0.51){\vector(0,1){0.00001}}
\put(0.3,0.5){\circle{0.14}}\put(0.37,0.51){\vector(0,1){0.00001}}
\put(0.5,0.5){\circle{0.14}}\put(0.57,0.51){\vector(0,1){0.00001}}

\put(0.1,0.7){\circle{0.14}}\put(0.17,0.71){\vector(0,1){0.00001}}
\put(0.7,0.7){\circle{0.14}}\put(0.77,0.71){\vector(0,1){0.00001}}
\put(0.9,0.7){\circle{0.14}}\put(0.97,0.71){\vector(0,1){0.00001}}
\put(0.3,0.7){\circle{0.14}}\put(0.37,0.71){\vector(0,1){0.00001}}
\put(0.5,0.7){\circle{0.14}}\put(0.57,0.71){\vector(0,1){0.00001}}


\end{picture}
\caption{The Hasse diagram for a finite bi-dimensional grid and the oriented independent elementary cycles
associated to the elementary faces of the lattice. } \label{fig3}
\end{figure}
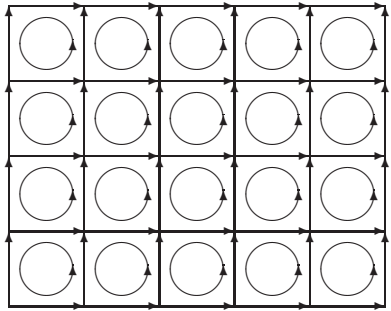
The problem of establishing wether $\mu_1\preceq \mu_2$ is equivalent
to the problem of determining if a polyhedron in dimension $NM$ identified by $2NM+M+N$ inequalities is empty or not. The inequalities are one for each edge of the Hasse diagram. Let $\alpha\left(x_1,x_2\right)$ be the real variable associated to the elementary cycle centered in $\left(x_1+\frac 12,x_2+\frac 12\right)$. Let also $\phi^*$ be a solution of \eqref{bb2}.
We can for example consider
\begin{equation}\label{2dfi}
\left\{
\begin{array}{l}
\phi^*\big((x_1,x_2),(x_1+1,x_2)\big)=\frac 12 \big(\left[F_1-F_2\right](x_1,x_2)-\left[F_1-F_2\right](x_1,x_2-1)\big)\,,\\
\phi^*\big((x_1,x_2),(x_1,x_2+1)\big)=\frac 12 \big(\left[F_1-F_2\right](x_1,x_2)-\left[F_1-F_2\right](x_1-1,x_2)\big)\,,
\end{array}
\right.
\end{equation}
where for $i=1,2$,  $F_i(x_1,x_2):=\sum_{(y_1,y_2)\leq (x_1,x_2)}\mu_i(y_1,y_2)$ is the distribution function of $\mu_i$.
The inequalities can be summarized by
\begin{equation}
\left\{
\begin{array}{l}
\frac{\left[F_1-F_2\right](x_1,x_2)-\left[F_1-F_2\right](x_1-1,x_2)}{2}\geq \alpha(x_1,x_2)-\alpha(x_1-1,x_2)\,,\\
\frac{\left[F_1-F_2\right](x_1,x_2)-\left[F_1-F_2\right](x_1,x_2-1)}{2}\geq \alpha(x_1,x_2-1)-\alpha(x_1,x_2)\,,
\end{array}
\right.\label{maurocalibani}
\end{equation}
that have to be satisfied for any vertex $(x_1,x_2)$ of the grid.
Clearly, when in \eqref{maurocalibani} it appears a variable $\alpha$ associated to an elementary face outside of the grid we mean that its value is zero.

A complete characterization of when inequalities \eqref{maurocalibani} determine a non empty polyhedron is difficult but it can be given in some special cases like for example a strip ($M=1$). However it is easy to find sufficient conditions to have $\mu_1\preceq \mu_2$. For example choosing $\alpha=\pm\frac 12\left[F_1-F_2\right]$ we deduce that if $F_1-F_2$ is  increasing in one of the two coordinates then $\mu_1\preceq \mu_2$.

A similar scheme can be developed for planar posets.

\subsection{A generalized construction}
A very general and much used construction of a compatible coupling is obtained considering two functions $G_1$ and $G_2$ defined on a set $\Omega$, taking values on the poset $(V,\leq)$ and such that $G_1(\omega)\leq G_2(\omega)$ for any $\omega\in \Omega$. Given a random variable $U$ taking values on $\Omega$ the joint law $\rho$ of the random variables $(X_1,X_2)=(G_1(U),G_2(U))$ is a compatible coupling between the distribution $\mu_1$ of $X_1$ and the distribution $\mu_2$ of $X_2$ so that $\mu_1\preceq\mu_2$. Using flows the argument is as well elementary. For any $x\leq y$ fix a path $\gamma_{x,y}$ on $(V,E)$. Then
$$\sum_{\omega\in\Omega} \mathbb P(U=\omega)Q_{\gamma_{G_1(\omega),G_2(\omega)}}=\sum_{x,y}\rho(x,y)Q_{\gamma_{x,y}}$$
is a flow with divergence coinciding with $\mu_1-\mu_2$. The coupling argument does not work if $G_1\not\leq G_2$ while the flow argument can still work generalizing this classic construction.

We illustrate the simplest possible approach that can be generalized in several ways. Recall that $(V,\mathcal E)$ is the un-oriented graph associated to $(V,E)$. To any pair $x,y\in V$ (not necessarily ordered) we associate a fixed path $\gamma_{x,y}$ in $(V,\mathcal E)$ going from $x$ to $y$. This is a sequence $(x_0,x_1, \dots x_n)$ such that $x_0=x$, $x_n=y$ and $\{x_i,x_{i+1}\}\in \mathcal E$. To any path $\gamma$ on $(V,\mathcal E)$ we associate the discrete vector field $\phi_\gamma$ defined by
\begin{equation}\label{fest}
\phi_\gamma(u,v):=\left\{
\begin{array}{ll}
1 & \textrm{if}\ (u,v)\in \gamma\\
-1 & \textrm{if}\ (v,u)\in \gamma\\
0 & \textrm{otherwise}\,.
\end{array}
\right.
\end{equation}
We consider the discrete vector field
\begin{equation}\label{evai}
\phi=\sum_{\omega\in\Omega}\mathbb P(U=\omega)\phi_{\gamma_{G_1(\omega),G_2(\omega)}}=\sum_{x,y}\rho(x,y)\phi_{\gamma_{x,y}}\,.
\end{equation}
If $\phi(x,y)\geq 0$ for any $(x,y)\in E$ then the flow on $(V,E)$ defined by $Q(x,y)=\phi(x,y)$ has divergence $\mu_1-\mu_2$ and we deduce $\mu_1\preceq \mu_2$. The basic idea is the following. It may happen that $G_1(U)\not\leq G_2(U)$ that corresponds to negative flows across some edges in $(V,E)$. Nevertheless the total net flow across each edge is positive and this is enough to prove $\mu_1 \preceq \mu_2$.

A simple illustrative case is the following. Let $U_1$ be a random variable taking values on $V=\mathbb Z^2\cap\left([0,N]\times[0,M]\right)$ and having distribution $\mu$. The random variable $U_2$ is obtained moving the random lattice point $U_1$ uniformly at random on one of its 4 nearest neighbors vertices of $\mathbb Z^d$. If this point is outside the rectangle $[0,N]\times[0,M]$ then $U_2=U_1$. We fix $U=(U_1,U_2)$ and $X_1=G_1(U_1,U_2)=U_1$ and $X_2=G_2(U_1,U_2)=U_2$. We have that the law of $\big(G_1(U_1,U_2),G_2(U_1,U_2)\big)$ is not a monotone coupling of $\mu_1=\mu$ and $\mu_2$ the law of $X_2$. The discrete vector field
\eqref{evai} is however
\begin{equation}
\label{evai2}
\left\{
\begin{array}{ll}
\phi\left(x, x+(1,0)\right)=\frac 14\left[\mu(x)-\mu(x+(1,0))\right] & x_1=0,\dots,N-1;\, x_2=0,\dots , N\,,\\
\phi\left(x, x+(0,1)\right)=\frac 14\left[\mu(x)-\mu(x+(0,1))\right] & x_1=0,\dots,N;\, x_2=0,\dots , N-1\,.
\end{array}
\right.
\end{equation}
We deduce immediately that if $\mu$ is decreasing (i.e. $\mu(x)\geq\mu(x+(1,0))$ and $\mu(x)\geq\mu(x+(0,1))$ for edges belonging to the rectangle) then the vector field is positive along increasing directions and $\mu_1\preceq \mu_2$.

\subsection{Product couplings and flows}\label{pcef}
Let $(V,\leq)$ be a finite poset with associated the Hasse diagram $(V,\underline E)$ and consider the product partial order on $V^N$. An element $\eta\in V^N$ is written as
$\eta=(\eta(1),\dots,\eta(N))$ and for $\eta, \xi\in V^N$ we write $\eta\leq\xi$ if
$\eta(i)\leq\xi(i)$ for any $i$.
Given $\eta\in V^N$, $x\in V$ and $i=1,\dots ,N$ we denote by
$$\eta^{i}_{x}=(\eta(1),\dots,\eta(i-1),x,\eta(i+1),\dots ,\eta(N))$$
the element of $V^N$
with $\eta(i)$ replaced by $x$. The Hasse diagram $(V^N,\underline{E}^N)$ for the product poset has a directed edge $(\eta,\xi)$ if and only if $\xi=\eta^{i}_{x}$ for some $i,x$ and $(\eta(i),x)\in \underline E$.

Let $\mu_1^i$ and $\mu_2^i$ $i=1,\dots,N$ be a collection of probability measures on $V$ such that for any $i$ we have $\mu_1^i\preceq \mu_2^i$.

Then for the product measures we have $\bigotimes_{i=1}^N\mu_1^i\preceq \bigotimes_{i=1}^N\mu_2^i$: indeed, if $\rho^i$ is a monotone coupling between $\mu_1^i$ and $\mu_2^i$ then $\bigotimes_{i=1}^N\rho^i$ is a monotone coupling between $\bigotimes_{i=1}^N\mu_1^i$ and $\bigotimes_{i=1}^N\mu_2^i$.

Let us illustrate that there is an equivalent construction with flows. Let $Q^i$ be a flow on $(V,\underline E)$ such that $\div Q^i(x)=\mu_1^i(x)-\mu_2^i(x)$. Let us define
\begin{equation}\label{alfafin}
\gamma^i(\eta)=\left[\prod_{j<i}\mu_2^j(\eta(j))\right]\left[\prod_{j>i}\mu_1^j(\eta(j))\right]
\end{equation}
and observe that \eqref{alfafin} does not depend on $\eta(i)$. We define the flow $Q$ on $(V^N,\underline E^N)$ as
\begin{equation}\label{fluflu}
Q(\eta,\eta^{i}_{x})=\gamma^i(\eta)Q^i(\eta(i),x)\,.
\end{equation}
Then we have
\begin{eqnarray}
\div Q(\eta)&=&\sum_{i=1}^N\gamma^i(\eta)\sum_{y\in V}\left[Q^i(\eta(i),y)-Q^i(y,\eta(i))\right] \nonumber\\
&=&\sum_{i=1}^N\gamma^i(\eta)\left[\mu_1^i(\eta(i))-\mu_2^i(\eta(i))\right]\nonumber \\
&=&\prod_{i=1}^N\mu_1^i(\eta(i))-\prod_{i=1}^N\mu_2^i(\eta(i))\,
\label{ondra}
\end{eqnarray}
where the last equality follows from the special form \eqref{alfafin} since the sum in the second line of \eqref{ondra} is telescopic and only the initial and final terms survive.

\subsection{The k-generation poset}

We call k-generation poset a poset $(V, \leq )$ defined as follows. For simplicity we consider only the case $V$ finite.
The set $V$ is partitioned into $k$ disjoints subsets $V_1, \dots , V_k$. The elements of the subset $V_j$ are the elements of the generation $j$. If $x\in V_i$ and $y\in V_j$ then we have $x\leq y$ if $i<j$ while $x$ and $y$ are not comparable when $i=j$. An example of the corresponding Hasse diagram is drawn in Figure 4. Elements of the same generation are drawn in the same horizontal line. We now discuss conditions to have $\mu_1\preceq \mu_2$ on a k-generation poset.

\begin{figure}\label{el-hasse1}
\setlength{\unitlength}{1.5cm}
\begin{picture}(4,2.5)

\put(1,0){\circle*{0.1}}
\put(1.5,0){\circle*{0.1}}
\put(2,0){\circle*{0.1}}
\put(2.5,0){\circle*{0.1}}
\put(1,0.5){\circle*{0.1}}
\put(1.5,0.5){\circle*{0.1}}
\put(1,1){\circle*{0.1}}
\put(1.5,1){\circle*{0.1}}
\put(2,1){\circle*{0.1}}
\put(1,1.5){\circle*{0.1}}
\put(1.5,1.5){\circle*{0.1}}

\put(1,1){\vector(0,1){0.5}}
\put(1,1){\vector(1,1){0.5}}

\put(1.5,1){\vector(0,1){0.5}}
\put(1.5,1){\vector(-1,1){0.5}}

\put(2,1){\vector(-1,1){0.5}}
\put(2,1){\vector(-2,1){1}}

\put(1,0.5){\vector(0,1){0.5}}
\put(1,0.5){\vector(1,1){0.5}}
\put(1,0.5){\vector(2,1){1}}

\put(1.5,0.5){\vector(0,1){0.5}}
\put(1.5,0.5){\vector(1,1){0.5}}
\put(1.5,0.5){\vector(-1,1){0.5}}

\put(1,0){\vector(1,1){0.5}}
\put(1,0){\vector(0,1){0.5}}
\put(1.5,0){\vector(0,1){0.5}}
\put(1.5,0){\vector(-1,1){0.5}}
\put(2,0){\vector(-1,1){0.5}}
\put(2,0){\vector(-2,1){1}}
\put(2.5,0){\vector(-2,1){1}}
\put(2.5,0){\vector(-3,1){1.5}}
\end{picture}
\caption{The Hasse diagram of a 4-generation poset.}
\end{figure}
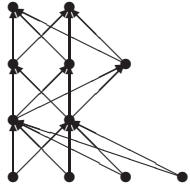

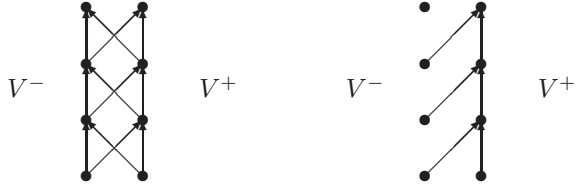
\begin{figure}\label{el-hasse2}
\setlength{\unitlength}{1.5cm}
\begin{picture}(4,2.5)

\put(-0.7,0.7){$V^-$}
\put(1.0,0.7){$V^+$}

\put(0,0){\circle*{0.1}}
\put(0.5,0){\circle*{0.1}}
\put(0,0.5){\circle*{0.1}}
\put(0.5,0.5){\circle*{0.1}}
\put(0,1){\circle*{0.1}}
\put(0.5,1){\circle*{0.1}}
\put(0,1.5){\circle*{0.1}}
\put(0.5,1.5){\circle*{0.1}}

\put(0,1){\vector(0,1){0.5}}
\put(0,1){\vector(1,1){0.5}}

\put(0.5,1){\vector(0,1){0.5}}
\put(0.5,1){\vector(-1,1){0.5}}

\put(0,0.5){\vector(0,1){0.5}}
\put(0,0.5){\vector(1,1){0.5}}

\put(0.5,0.5){\vector(0,1){0.5}}
\put(0.5,0.5){\vector(-1,1){0.5}}

\put(0,0){\vector(1,1){0.5}}
\put(0,0){\vector(0,1){0.5}}
\put(0.5,0){\vector(0,1){0.5}}
\put(0.5,0){\vector(-1,1){0.5}}

\put(3,0){\circle*{0.1}}
\put(3.5,0){\circle*{0.1}}
\put(3,0.5){\circle*{0.1}}
\put(3.5,0.5){\circle*{0.1}}
\put(3,1){\circle*{0.1}}
\put(3.5,1){\circle*{0.1}}
\put(3,1.5){\circle*{0.1}}
\put(3.5,1.5){\circle*{0.1}}
\put(2.3,0.7){$V^-$}
\put(4,0.7){$V^+$}

\put(3,1){\vector(1,1){0.5}}

\put(3.5,1){\vector(0,1){0.5}}

\put(3,0.5){\vector(1,1){0.5}}

\put(3.5,0.5){\vector(0,1){0.5}}

\put(3,0){\vector(1,1){0.5}}
\put(3.5,0){\vector(0,1){0.5}}

\end{picture}
\caption{Two reduced versions of the Hasse diagram of the 4-generation poset in Figure 4. Left columns correspond to vertices of the type $V^-$ while right columns to vertices of the type $V^+$. }
\end{figure}
Let $V^-_i:=\left\{x\in V_i\,:\,\mu_1(x)-\mu_2(x)\geq 0\right\}$ and $V^+_i:=\left\{x\in V_i\,:\,\mu_1(x)-\mu_2(x)<0 \right\}$. Let us also introduce a reduced k-generation poset having on each generation exactly 2 vertices (Figure 5 on the left). We label also the two vertices of the generation $i$ of this reduced poset with the same symbols $V^\pm_i$. Let us call
$$
\left\{
\begin{array}{l}
\tilde\mu_1(V_i^+)-\tilde\mu_2(V_i^+):=\sum_{x\in V^+_i}\left[\mu_1(x)-\mu_2(x)\right]\,,\\
\tilde\mu_1(V_i^-)-\tilde\mu_2(V_i^-):=\sum_{x\in V^-_i}\left[\mu_1(x)-\mu_2(x)\right]\,.
\end{array}
\right.
$$
First of all we claim that the existence of a flow on the original poset having divergence $\mu_1-\mu_2$ is equivalent to the existence of a flow having divergence $\tilde\mu_1-\tilde\mu_2$ on the reduced poset. The key idea to prove this fact is the decomposition \eqref{finite-decomp2} for which we recall that all the paths $\gamma_n$ exit from elements of $V^-$
and end on elements of $V^+$. It is possible to construct a flow on the original poset splitting the paths of the decomposition of the flow on the reduced poset and conversely it is possible to construct a flow on the reduced poset gluing together paths of the decomposition of the flow on the original poset. The details can be easily fixed and we will not discuss them.

Second we claim that the existence of a flow having divergence $\tilde\mu_1-\tilde\mu_2$ on the reduced poset (figure 5 left) is equivalent to the existence of a flow having the same divergence on the further reduced poset on the right of Figure 5. The idea behind the proof of this statement is the same of the previous one.
The only non trivial fact is that if there exists a flow for the poset on the left hand side of figure 5 then there exists also a flow for the poset on the right hand side.
Recall that on both sides of Figure 5 the vertices on the left columns are characterized by the condition $\tilde\mu_1-\tilde\mu_2\geq 0$. Consider a path $\gamma$ in the decomposition of the flow for the poset on the left. This path will start from a vertex on the left column and will finish on a vertex on the right column. We modify this path keeping fixed the starting and final vertices. In particular the modified path will jump immediately on the right column and will continue moving on that side. This new path is a path also for the further reduced poset on the right hand side of Figure 5. The flow obtained as a superposition of all the paths modified this way will have the same divergence of the original one and moreover will be a flow for the reduced poset.

Since the further reduced poset on the right hand side of Figure 5 is a tree we can apply the results of section  \ref{fait}. We claim that $\mu_1\preceq\mu_2$ on a k-generation poset if and only if
\begin{equation}
\sum_{i=1}^{j-1}\sum_{x\in V_i}\left(\mu_1(x)-\mu_2(x)\right)-\sum_{x\in V_j}\left[\mu_1(x)-\mu_2(x)\right]_-\geq 0\,, \qquad j=1,\dots ,k\,. \label{k-gen}
\end{equation}
In the above formula $[\cdot]_-$ denotes the negative part and for $j=1$ the sum over an empty set of indices has to be interpreted as zero. Indeed recalling that the left column is characterized by $\tilde\mu_1-\tilde\mu_2\geq 0$ it is easy to see that condition \eqref{non} is automatically satisfied on all the edges going from one column to the other. The remaining conditions can be compactly written like \eqref{k-gen}.

\subsection{Lattices}

For an integer $N\geq 2$, consider the lattice $V:=\{0,1\}^N$ with the usual partial order. An element $\eta\in V$
has the form $\eta=\left(\eta(1),\dots ,\eta(N)\right)$ with
$\eta(i)\in \left\{0,1\right\}$. The Hasse diagram associated is $(V,\underline E)$ where
$(\eta,\eta')\in \underline E$ if and only if $\eta'$ is obtained by $\eta$ changing
one coordinate of $\eta$ from $0$ to $1$. The digraph on the right of Figure 2
is the Hasse diagram for this poset when $N=2$.

A classic sufficient condition to have stochastic monotonicity is the Holley condition \cite{Ho}. An alternative sufficient condition to have $\mu_1\preceq \mu_2$ is discussed in \cite{Hos}. This condition is very simple and natural and coincides with the fact that $\mu_1-\mu_2$ is not increasing. Differently from the Holley condition depends just on the difference between the two measures and can be easily proved using flows.
The proof in \cite{Hos} is also elementary but it assumes Holley result while our proof is completely independent. The result is strictly related to the geometry of the poset. It is possible indeed to construct posets (with Hasse diagram being a tree for example) for which $\mu_1-\mu_2$ is not increasing but nevertheless $\mu_1\not\preceq \mu_2$.

Our proof is by induction. For $N=2$ a necessary and sufficient condition to have $\mu_1\preceq\mu_2$ is
\eqref{el-simple}. It is easy by a direct inspection to check that if $\mu_1-\mu_2$ is non increasing then \eqref{el-simple} holds.
Let us now assume the validity of the result for $N$.
Any $\eta\in \left\{0,1\right\}^{N+1}$ is of the form $(\tilde\eta,0)$ or $(\tilde\eta, 1)$ with $\tilde\eta\in\left\{0,1\right\}^{N}$. We define
\begin{equation}\label{voltura}
\left(\tilde\mu_1-\tilde\mu_2\right)(\tilde\eta)
:=\frac{\left[\mu_1(\tilde\eta,1)-\mu_2(\tilde\eta,1)\right]+\left[\mu_1(\tilde\eta,0)-\mu_2(\tilde\eta,0)\right]}{2}\,.
\end{equation}
Since $\mu_1-\mu_2$ is non increasing on $\{0,1\}^{N+1}$ then also $\tilde\mu_1-\tilde\mu_2$ is non increasing on $\{0,1\}^N$. By induction, there exists a flow $\tilde Q$ on the Hasse diagram of the poset of order $N$ such that $\div\tilde Q=\tilde\mu_1-\tilde\mu_2$. Now, consider the flow $Q$ on the Hasse diagram of the poset of order $N+1$ defined as follows: for any   directed edge $(\tilde \eta,\tilde \xi)$ of the Hasse diagram of order $N$ we pose
$$
Q((\tilde \eta,1),(\tilde\xi,1))=Q((\tilde \eta,0),(\tilde\xi,0))=\tilde Q\left(\tilde \eta,\tilde\xi \right)\,,
$$
while, for all the remaining edges we fix $Q(\eta,\eta')=0$. We have
\begin{equation}\label{conf}
\div Q(\tilde\eta, 1)=\div Q(\tilde\eta,0)=\div \tilde Q\left(\tilde\eta\right)=\left(\tilde\mu_1-\tilde\mu_2\right)(\tilde\eta)\,.
\end{equation}
We now define another flow $Q^a$ on the Hasse diagram of order $N+1$ as follows: for any $\tilde\eta\in\{0,1\}^N$ we pose
\begin{equation}\label{carusi}
Q^a((\tilde\eta,0),(\tilde\eta,1)):
=\frac{\left[\mu_1(\tilde\eta,0)-\mu_2(\tilde\eta,0)\right]-\left[\mu_1(\tilde\eta,1)-\mu_2(\tilde\eta,1)\right]}{2}\,.
\end{equation}
We fix then $Q^a(\eta,\eta')=0$ for all the remaining edges. This is a well defined flow since the right hand side of \eqref{carusi} is non-negative being $\mu_1-\mu_2$ not increasing. Using \eqref{conf} and \eqref{carusi}  we obtain that the flow $Q+Q^a$ satisfies $\div (Q+Q^a)=\mu_1-\mu_2$ and the proof is complete.


\end{document}